%% file: BourgainCounterexample_Final_submit_ArXiv_Oct2020.tex
\documentclass[oneside,11pt]{amsart}
\usepackage{amssymb,latexsym,amsmath,amsthm,enumitem,mathrsfs}
\usepackage[margin=1in]{geometry}
\usepackage{hyperref}
\usepackage{fancyhdr}
\usepackage{color}
\usepackage{stmaryrd}
\usepackage{mathrsfs}
\usepackage{lineno}

\input{format}

\newcommand{\exendnote}[1]{}

\newcommand{\apphide}[1]{}

\begin{document}

\title[]{On Bourgain's counterexample for the Schr\"odinger maximal function}

\author[Pierce]{Lillian B. Pierce}
\address{Department of Mathematics, Duke University, 120 Science Drive, Durham NC 27708 USA }
\email{pierce@math.duke.edu}

\begin{abstract}
This paper provides a rigorous derivation of a counterexample of Bourgain, related to a well-known question of pointwise a.e. convergence for the solution of the linear Schr\"odinger equation, for initial data in a Sobolev space. This counterexample combines ideas from analysis and number theory, and the present paper demonstrates how to build such counterexamples from first principles, and then optimize them. 

\end{abstract}

\maketitle 
\begin{center}
\emph{Dedicated to the memory of Jean Bourgain}
\end{center}

\section{Introduction}

This paper provides a rigorous explanation of a criterion established  by Bourgain \cite{Bou16}, concerning the 
solution to the linear Schr\"odinger equation,
\[
\begin{cases}
i\partial_t u - \Del u = 0, \quad (x,t) \in \R^n \times \R, \\
u(x,0) = f(x), \quad x \in \R^n,
\end{cases}
\]
which is given for an appropriate initial data function $f$ (of Schwartz class for example) by
\[( e^{it \Del} f)(x) = \frac{1}{(2\pi)^n} \int_{\R^n} \hat{f}(\xi) e^{ i( \xi \cdot x + |\xi|^2 t)} d\xi.\]
A central question of Carleson \cite{Car80} asks for the optimal value of $s$ for which it is true that for all functions $f$ belonging to the Sobolev space $H^s(\R^n)$, the pointwise convergence result
\beq\label{ptwise}
 \lim_{t \maps 0} (e^{it \Del} f)(x) =f(x) 
 \eeq
holds for almost every $x \in \R^n$.
In dimension $n=1$,  Carleson proved it is sufficient to have $s \geq 1/4$ \cite[Eqn (14) p. 24]{Car80}, and this was shown to be necessary by Dahlberg and Kenig \cite{DahKen82}, thus resolving the one-dimensional case. In dimensions $n\geq 2$, the problem was studied by many authors, but remained open until 2019.
We only mention a few very recent highlights in the literature.
Lee \cite{Lee06} used bilinear techniques to show that in dimension $n=2$, $s> 3/8$  suffices  to guarantee pointwise a.e. convergence; Bourgain \cite{Bou13} then used multilinear techniques to prove  that for any dimension $n$, $s > 1/2 -1/(4n)$ suffices.
Also in \cite{Bou13}, Bourgain improved the necessary condition, writing that ``perhaps the most interesting point in this note is a disproof of what one seemed to believe, namely that $f \in H^s(\R^n),$ $s>1/4$, should be the correct condition in arbitrary dimension $n$.'' Precisely, Bourgain showed in that paper that $s \geq 1/2 -1/n$ is necessary,  using the distribution of lattice points on spheres.
Soon after, Luc\`{a} and Rogers  improved on this, showing that   $s \geq 1/2 - 1/(n+2)$ is necessary via a counterexample involving an ergodicity argument; this appeared in  \cite{LucRog19a}.  
Subsequently, an alternative argument for this  condition,  via pseudoconformal transformations, was given by Demeter and Guo  (see the preprint \cite{DemGuo16}).
 
Our focus is on the 2016 work of
Bourgain \cite{Bou16}, which proved via a counterexample construction that for any $n \geq 2$, the pointwise a.e. convergence (\ref{ptwise}) can fail if $s < \frac{n}{2(n+1)}$. Bourgain's acclaimed work furthermore suggested that $s \geq \frac{n}{2(n+1)}$ could be the optimal range for a positive result on pointwise convergence. Soon after, in dimension $n=2$, Du, Guth and Li \cite{DGL17} proved it is sufficient to have $s>1/3$, resolving all but the endpoint case in this dimension. In \cite{DGLZ18}, Du, Guth, Li and Zhang proved  $s > \frac{n+1}{2(n+2)}$ suffices. 
Finally,  landmark work of Du and Zhang \cite{DuZha19} resolved all but the endpoint cases  for all $n \geq 3$, proving that   $s > \frac{n}{2(n+1)}$ suffices in all dimensions.

Bourgain's influential counterexample  combined ideas from Fourier analysis and analytic number theory.
We recall the precise statement of \cite[Prop. 1]{Bou16}.
\begin{thm}[Bourgain]\label{thm_Bou}
Fix $n \geq 2$ and  $s < \frac{n}{2(n+1)}$.  There exists a sequence of real numbers $R_k \maps \infty$ as $k \maps \infty$, and a sequence of functions $f_k \in L^2(\R^n)$ such that $\|f_k\|_{L^2(\R^n)}=1$ and $\hat{f}_k$ is supported in an annulus $ \{(1/C)R_k \leq |\xi| <CR_k\}$, such that  
\beq\label{Bou_est}
 \lim_{k \maps \infty} R_k^{-s} \| \sup_{0<t<1} |e^{it \Del} f_k(x)| \|_{L^1(B_n(0,1))} = \infty.
 \eeq
 \end{thm}
 \noindent
To recall why this result implies the failure of (\ref{ptwise}) for such $s$, see Appendix A.

Bourgain's original treatment \cite{Bou16} provided a skeletal overview of the construction of the functions $f_k$.
Our aim is to flesh out these ideas, providing not only a rigorous derivation of Theorem \ref{thm_Bou}, but also an animation of how to build a counterexample from first principles.

We  construct Bourgain's counterexample and prove Theorem \ref{thm_Bou} in  three stages: first we examine the basic construction of a test function $f$ as a product of smooth one-variable functions that have been scaled and modulated. Second, we construct our ultimate test function $f$ as a sum of such functions so as to introduce arithmetic behavior to $(e^{it \Delta}f)(x)$. Third, we construct a set of $x$ for each of which a corresponding value of $t$ may be chosen so that this arithmetic behavior can be evaluated precisely, in the form of a Gauss sum, leading to a lower bound for $|(e^{it \Delta}f)(x)|$.  

 To initiate our discussion, we start with generic parameters. As the argument proceeds, we will have to assume various constraints on the parameters, and ultimately we will rigorously determine an optimal choice of parameters under these constraints. In particular, this   will clearly reveal the fundamental limitation of Bourgain's construction (which is   confirmed to be optimal, up to the endpoint, by the positive results of \cite{DuZha19}).
We anticipate that this ``handbook'' of the relevant ideas   at the intersection of analysis and number theory will be   useful for future work on the many remaining open problems in the area.
 
For example, we mention five possible directions of current interest, which motivate the present exposition.
 After Bourgain's work \cite{Bou16}, Luc\`a and Rogers \cite{LucRog19} provided a different counterexample construction to also recover the necessity of $s \geq n/(2(n+1))$,
  via  ergodicity arguments. Along with their earlier work \cite{LucRog19a}, this  importantly also extends to the study initiated in \cite{SjoSjo89} of divergence on sets of lower-dimensional Hausdorff measure; see for example \cite{LucRog17,LucRog19} for open questions. 
Recently, \cite{DKWZ19}
used Bourgain's counterexample as a ``black box'' input for a construction that shows that the local estimate
\[ \| \sup_{0<t<1} |(e^{it \Del} f) (x)|\; \|_{L^p(B_n(0,1))}  \leq C_s \|f\|_{H^s(\R^n)} \]
for all $s > \frac{n}{2(n+1)}$ can fail if $p> 2 + \frac{4}{(n-1)(n+2)}$. This raises an open question,  also stated in \cite{DuZha19}:
to determine the optimal $p=p(n)$ for which this local estimate holds for all $s > \frac{n}{2(n+1)}$, and to identify the optimal $s=s(n,p)$ for which the local estimate holds for a fixed $p>2$. 
In another direction, \cite{CFW18} studies the rate of pointwise convergence, for   $s$ such that a.e. convergence occurs in (\ref{ptwise}).
Next,   \cite{DemGuo16} asks for the sharp value of  $s$ for pointwise convergence questions related to other curved
hypersurfaces $(\xi,\phi(\xi)) \subseteq \R^{n+1}$, generalizing the paraboloid $(\xi,|\xi|^2)$. Initial positive results for one  such class of $\phi$ have been obtained in \cite{ChoKo18x}. This direction also relates to a broad class of open questions posed by Bourgain  \cite[\S 5]{Bou13} in the context of maximal functions associated to oscillatory integral operators. Finally, there are the corresponding questions  in the periodic case; see for example   \cite{MoyVeg08}.

\subsection{Acknowledgements}
When Bourgain's counterexample came out, a number of people contacted me with questions about how it worked. This note answers those questions,  and moreover explains how one would naturally \emph{arrive} at this construction, and optimize it. This note is intended to be accessible to a broad audience, and to give an appreciation of Bourgain's view of this problem, connecting analysis and number theory.
I thank Valentin Blomer, Renato Luc\`a, Keith Rogers, Ruixiang Zhang, and the referee for a number of  helpful comments. I also thank Po-Lam Yung and Jongchon Kim for many insightful suggestions and corrections to an earlier draft, and additionally Kim for contributions to Appendix A.

\subsection{Notation}
We denote by $B_m(c,r)$  the Euclidean ball in $\R^m$, centered at $c$ and of radius $r$. We use the conventions   that $e(x) = e^{ix}$ and $\hat{f}(\xi) = \int_{\R^m} f(x) e^{-i x \cdot \xi} dx$, so that correspondingly $f(x) = (2\pi)^{-m} \int_{\R^m} \hat{f}(\xi)e^{ i x \cdot \xi} d\xi$ and Plancherel's theorem states $\|f\|_{L^2(\R^m)}^2 = (2\pi)^{-m}\|\hat{f} \|_{L^2(\R^m)}^2$. For an appropriately smooth and sufficiently decaying function $\Phi$ on $\R^m$ (for example of Schwartz class), for any shift $M \in \R^m$ and any scaling factor $S>0$,  
\[[\Phi (Sx)e(M\cdot x)]\hat{\;}(\xi) =  \frac{1}{S^m}\hat{\Phi}\left(\frac{\xi - M}{S}\right).\]
Thus if  $\hat{\Phi}$ is supported in $B_m(0,1)$ the Fourier transform of $ \Phi(Sx)e(M\cdot x)$ is supported in $B_m(M,S)$. 
By Plancherel's theorem,
\[ \| \Phi (Sx)e(M\cdot x) \|_{L^2(dx)} = (2\pi)^{-m/2} \|[\Phi (Sx)e(M\cdot x)]\hat{\;}(\xi)\|_{L^2(d\xi)} = S^{-m/2} \|\Phi\|_{L^2}.\]
It will be convenient to scale each variable independently, and thus for $S \in \R^m_{>0}$ we define $S \circ x = (S_1x_1,\ldots, S_mx_m)$, and let $S^{-1} = (S_1^{-1},\ldots, S_m^{-1})$ and $\|S\| = \prod S_j$. Then the Fourier transform of the function $\Phi (S\circ x)e(M\cdot x)$ is $\|S\|^{-1} \hat{\Phi}(S^{-1} \circ (\xi-M))$ and the  $L^2$ norm is $\|S\|^{-1/2} \|\Phi\|_{L^2}$.

\section{The basic motivating construction}\label{sec_basic}

We record the version of  Theorem \ref{thm_Bou} that we prove, as follows.
\begin{thm}\label{thm_restate}
Let $n \geq 2$ and $s>0$, and suppose that there is a constant $C_s$ such that for all   $f \in H^s(\R^n)$,
\beq\label{Bou'}
\| \sup_{0<t<1} | e^{it \Del}f| \; \|_{L^1(B_n(0,1))} \leq C_s \|f \|_{H^s(\R^n)}.
\eeq
Then $s \geq \frac{n}{2(n+1)}.$
\end{thm}
It suffices to prove that for each $s< \frac{n}{2(n+1)}$ we can construct a sequence $\{ f_k \}$  such that 
\[ 
\lim_{k \maps \infty}  \frac{ \| \sup_{0<t<1} | e^{it \Del}f_k| \; \|_{L^1(B_n(0,1))}}{ \|f_k\|_{H^s(\R^n)}}=\infty.\]
Recall that the Sobolev space $H^s(\R^n)$ (Bessel potential space)  is the class of functions 
such that $(1-\Del)^{s/2}f$ lies in $L^2(\R^n)$, or equivalently such that $G_{-s} * f \in L^2(\R^n)$, 
where the Bessel kernel $G_{-s}$ is defined according to its Fourier transform  
$\hat{G}_{-s}  (\xi) = (1 + |\xi|^2)^{s/2}$.
     Plancherel's theorem shows that
\[ \| f \|_{H^s(\R^n)}^2 = \| G_{-s}  * f\|_{L^2(\R^n)}^2 = (2\pi)^{-n} \| \hat{G}_{-s} \hat{f}\|_{L^2(\R^n)}^2 = (2\pi)^{-n} \int_{\R^n} (1 + |\xi|^2)^{s} |\hat{f}(\xi)|^2 d\xi.\]
In particular if $\hat{f}$ is supported in the annulus $\{ R/C \leq |\xi| <CR\}$ for a constant $C>1$ then for every $R \geq C^{-1}$,
\[ 
  C^{-s} R^{s}\|f\|_{L^2(\R^n)} \leq \| f \|_{H^s(\R^n)}  \leq   2^{s/2} C^{s} R^{s}\|f\|_{L^2(\R^n)}.
\]
Thus it suffices to show that for every $s < \frac{n}{2(n+1)}$, there exist constants $C=C(n)$, $A_s =A(s,n)$ and $R_0 = R_0(s,n)$ and a value $s'>s$ such that the following holds: for each integer $R \geq R_0$ there exists a  function $f_R,$ with $\hat{f}_R$  supported in an annulus $A_n(R,C) :=\{R/C \leq  |\xi| < CR\}$, such that 
\beq\label{Bou''}
   \frac{ \| \sup_{0<t<1} | e^{it \Del}f_R| \; \|_{L^1(B_n(0,1))}}{R^{s'} \|f_R \|_{L^2(\R^n)}}\geq A_s   .
\eeq
Then in particular, given any constant $C_s$ we can choose $R$ sufficiently large that the corresponding   function $f_R$ violates (\ref{Bou'}), as desired.
To prove (\ref{Bou''}) for a function $f_R$, it suffices to construct a set $\Omega^*$ in $B_n(0,1)$
with positive measure (independent of $R$) such that for each $x$ in the set, there exists some $t \in (0,1)$ for which 
\beq\label{goal_tilde_f}
\frac{ |(e^{it \Del} f_R)(x)| }{\|f_R\|_{L^2(\R^n)}} \geq A_s R^{s'}.
  \eeq
The reader can think of this as our goal, although the set $\Omega^*$ we construct will have a small dependence on $R$, and thus we will formally prove (\ref{Bou''}).

\subsection{The basic construction}
We now fix $R \geq 1$. To begin our construction of an appropriate function $f = f_R$, we let $\phi$ be a  Schwartz function on $\R$ that takes non-negative values, and such that $\phi(0)  =\frac{1}{2\pi} \int \hat{\phi}(\xi)d\xi = 1$ and  $\hat{\phi}$ is supported in $[-1,1]$. 
For $x \in \R^n$ we define $\Phi_n(x) = \prod_{i=1}^n \phi(x_i)$.  Since $\phi$ is fixed once and for all, any constants will be allowed to depend on $\phi$.
(Note: to construct such a function $\phi$, let $\psi \in C_0^\infty(B_1(0,1/4))$ be such that $\frac{1}{2\pi}\int \psi(\xi) d\xi=1$. 
Then define $\phi$ according to $\hat{\phi} =\frac{1}{2\pi} \psi * \overline{\psi(- \cdot)}$, so that $\phi = |\check{\psi}|^2$, in which $\check{\psi} (x)= \frac{1}{2\pi} \int \psi (\xi) e^{i x \xi}d\xi$.)

We wish for $\hat{f}$ to be supported in an annulus $A_n(R,C) :=\{R/C \leq  |\xi| < CR\}$ for some fixed $C=C(n)>1$. 
 It is natural to begin with a candidate function of the shape
\beq\label{f_try1}
 f(x) =    \Phi_n(S \circ x) e(M\cdot x)
 \eeq
 for some $M \in \R^n$ and $S \in \R_{>0}^n$.  Temporarily let $\mathcal{B}$ denote the box $\prod [-S_j,S_j]$ so that $\hat{f}$ is supported in $\mathcal{B}+M$. If each coordinate of $M$ is about of size $R$ and each $S_j$ is an order of magnitude smaller, this support will be contained in an appropriate annulus. Precisely, we suppose each $M_j$ satisfies $R \leq M_j < 2R$ and $S^* = \max_j S_j \leq R^\sig$ for some $\sig< 1$. Then $\mathcal{B}+M \subset B_n(0, \sqrt{n} \cdot 2R+ \sqrt{n} S^*)\setminus B_n(0,\sqrt{n} R-\sqrt{n}S^*)$, so that once $n,\sig$ are fixed, there exists $R_1 = R_1(n,\sig)$ such that for all $R \geq R_1$, $\mathcal{B}+M \subset A_n(R,4\sqrt{n})$ for all such $M$.

For $f$ as defined above we have 
\begin{align}
 ( e^{it \Del} f)(x) & = \frac{1}{(2\pi)^n}\int_{\R^n}   \hat{\Phi}_n(\xi) e( (S \circ \xi + M)\cdot x
	 +  |S \circ \xi + M|^2 t)   d \xi  \nonumber \\
	 &  = e(M \cdot x + |M|^2 t)  \frac{1}{(2\pi)^n}\int_{\R^n}  \hat{\Phi}_n(\xi) e ( \xi \cdot (S\circ (x + 2Mt)) +  |S \circ \xi|^2 t ) d\xi. \label{integral_prod}
	 \end{align}
 We notice that  if $t$ is very small so that the term that is quadratic in $\xi$ is very small, then the integral should be well-approximated by an integral with linear phase, which we can evaluate precisely using
\beq\label{Ft_value}
  \frac{1}{(2\pi)^n}\int_{\R^n}  \hat{\Phi}_n(\xi) e ( \xi \cdot (S\circ (x + 2Mt))) d\xi =  \Phi_n(S\circ (x+2Mt)).
 \eeq
Since we constructed $\Phi_n$ so that $ \Phi_n(0)=1$, if we choose $S,M,x,t$ so that $S\circ (x+2Mt)$ is sufficiently close to the origin, by continuity we can give a lower bound  $ \Phi_n(S\circ (x+2Mt)) \geq 1- c_0$ for a  small $c_0>0$ of our choice. We also notice that the isolation of the factor $e(M \cdot x + |M|^2 t)$ in (\ref{integral_prod}) could allow us to utilize Diophantine properties of $x,t$. Of course on its own this factor has norm one, but instead of defining  $f$ as in (\ref{f_try1}), we could define $f$ as a finite number of summands of the form (\ref{f_try1}) for certain values of $M \in \Z^n$, and then in place of $e(M \cdot x + |M|^2 t)$ we would have an exponential sum, which we could evaluate. 

In the remainder of this section, we make these ideas rigorous for a single function defined by (\ref{f_try1}), by first justifying the approximation allowing us to reduce to (\ref{Ft_value}), which will also motivate our choice of the scaling parameter $S$, and will begin to refine our choices for $x$ and $t$. 
Motivated by this discussion,  in the next section we will re-define $f$ as a finite sum of terms like (\ref{f_try1}), which will allow us to take advantage of number-theoretic properties of exponential sums.

\subsection{Removal of the quadratic phase}
 By construction,   the integral in (\ref{integral_prod}) factors, so we can work one dimension at a time. Our simple tool is the following fact: the integral of a function $\mu$ weighted by $e(h(t))$ can be well-approximated by the integral of $\mu$ alone, as long as the derivative of $h$ is sufficiently small.  

\begin{lemma}\label{lemma_int_approx}
Let $a<b$ be fixed real numbers. Let $\mu$ be an integrable function on $\R$, and let $h$ be a real-valued $C^1$ function on $\R$.
Then
\[\int_a^b \mu(t) e(  h(t))dt = e(h(b)) \int_a^b \mu(t)   dt + E  \]
where 
\[ |E| \leq \| \mu\|_{L^1[a,b]} \|h'\|_{L^\infty[a,b]} \cdot (b-a). \]  
\end{lemma}
This follows from integration by parts, since
\[\int_a^b \mu(t) e(  h(t))dt = e(h(b)) \int_a^b \mu(y) dy - i \int_a^b (\int_a^t \mu(y)  dy) h'(t) e(h(t)) dt.\]
Fix $1 \leq j \leq n$ and apply the lemma to the $\xi_j$-th integral in  (\ref{integral_prod}), obtaining
\beq\label{int_1}
\frac{1}{2\pi} \int_{-1}^1 \hat{\phi}(\xi_j)e(\xi_jS_j(x_j+2M_jt)) e( S_j^2\xi_j^2 t)d\xi_j
	 = e(S_j^2  t)\phi(S_j(x_j+2M_jt))  + E
	 \eeq
with $|E| \leq (4/2\pi) \|\hat{\phi}\|_{L^1} S_j^2t \leq   \|\hat{\phi}\|_{L^1} S_j^2t$. Since $\phi(0)=1$ and $\phi$ is smooth, given any small $0<c_0 < 1/2$ there exists a $\del_0(c_0) \leq 1$ (depending on $\phi$) such that for any $\del_0 \leq \del_0(c_0)$, for all $|y| \leq \del_0$ we have $\phi(y) \geq 1-c_0/2$. 
Thus given $x_j$, if we choose $t$ such that $t=-x_j/(2M_j) + \tau$ with $|\tau| \leq \del_0/(2S_jM_j)$, and also $t \leq c_0/(4 \|\hat{\phi}\|_{L^1} S_j^2)$, then by (\ref{int_1}),
\beq\label{int_factor_j_inverse}
 |\frac{1}{2\pi} \int_{-1}^1 \hat{\phi}(\xi_j)e(\xi_j(S_j(x_j+2M_jt)) + S_j^2\xi_j^2 t)d\xi_j| \geq 1-c_0.
 \eeq
 In order for the two conditions on $t$ to be compatible,  we learn that 
 $x_j/(2M_j)$ and $\del_0/(2S_jM_j)$ must each be no bigger than $c_0/(8\|\hat{\phi}\|_{L^1}S_j^2)$. From this, we learn that
 we should focus on $x_j$ in a small neighborhood of the origin, say 
 \beq\label{x_range}
 |x_j| \leq c_1 < \del_0/2
 \eeq
  (with $c_1$ chosen appropriately, depending on $c_0, \phi$).
We also learn that we must have $S_j \leq M_j^{1/2}$, so that upon recalling that $R \leq M_j< 2R$, the largest we could take $S_j$ is of size $R^{1/2}$.

For one fixed coordinate $j$, for such $x_j$, we can thus choose $t$ and the parameters $M_j, S_j$ to justify (\ref{int_factor_j_inverse}).
But we would like to do so for all coordinates simultaneously. After $x_j$ is fixed, $t$ is constrained to a $\del_0/(2S_jM_j)$-neighborhood of $-x_j/(2M_j)$, so in particular, once $t$ is chosen to be compatible in this manner with $x_1$, in order for the same $t$ to also be compatible with $x_j$ for $j=2,\ldots, n$, the point $x$ would need to lie in a small set, of measure at most on the order of $\prod_{j=2}^n (S_jM_j)^{-1}$. This would force $x$ to lie in a set of measure at most $R^{-(n-1)}$, so with the goal of obtaining a set $x$ of positive measure   independent of $R$, we now make a different observation.

We return to the constraint that $|S_j(x_j + 2M_j t)| \leq \del_0$, which places the argument of $\phi(S_j(x_j + 2M_j t))$ sufficiently close to the origin. The issue we encountered above is that even if $x_j$ is small, a large value of $S_j$ places $S_jx_j$ far from the origin, so we must choose $t$ to cancel or nearly cancel this.  If instead $S_j=1$ and $|x_j| \leq \del_0/2$ then we only need $t \leq \del_0/(4M_j)$ for the constraint $|x_j + 2M_jt| \leq \del_0$ to be satisfied. This inspires us to take a hybrid approach: we will let $S_1=R^\sig$ for some $\sig \leq 1/2$ to be chosen later, and $t$ will be precisely constrained by $x_1$, but for $j=2,\ldots, n$, we will set $S_j=1$ so that $t$ is not precisely constrained by $x_j$.
To be concrete, we can choose 
\[ c_1 < \del_0/2 \leq 1/2, \qquad c_2 < 1/2\]
sufficiently small (depending on $c_0,  \phi$) such that the following holds:  fix $S_1=R^\sig$ for some $\sig \leq 1/2$ and $M_1=R$ and let $M_2,\ldots, M_n \in [R,2R)$ and assume that $x \in [-c_1,c_1]^n$ lies in a  small neighborhood of the origin. Choose  $t$ such that 
\beq\label{t_choice}
t=-x_1/(2R) + \tau \qquad \text{ with $|\tau| \leq c_2/S_1R$,}
\eeq
 in which case we also have  (by choosing $c_1,c_2$ appropriately small) that $|2 S_1 R \tau| \leq \del_0$ and
 \beq\label{t_fact1}
 t \leq c_0/(4 \|\hat{\phi}\|_{L^1} S_1^2), \qquad \text{and $\qquad t \leq \del_0/(8R) \leq \del_0/(4M_j)$ for each $j=2,\ldots,n$.}
 \eeq
We make one final restriction to ensure that $t \in (0,1)$: we require that $x_1  \in (-c_1,-c_1/2]$. 
Then we will have $t \in (0,1)$ as long as $c_1/2R + c_2/S_1R <1 $ and $c_1/(4R) > c_2/(S_1R)$, which will occur for all sufficiently large $R$, say $R \geq R_2=R_2(n,\phi,\sig)$.

The discussion above shows that with these constraints,
\[ \frac{ |( e^{it \Del}f)(x) | }{\|f\|_{L^2}} \geq  S_1^{1/2}|e(M \cdot x + |M|^2 t) |(1-c_0)^n = R^{\sig/2} (1-c_0)^n. \]
So far this is unsatisfactory, as it only shows (\ref{goal_tilde_f}) holds for $s < \sig/2$, which is no better than $s < 1/4$ (upon recalling $\sig \leq 1/2$). This only recovers the necessity of $s \geq 1/4$ for pointwise convergence of (\ref{ptwise}).
In order to improve on this, we take up our earlier point that we may want to construct $f$ as a sum of a finite number of terms like (\ref{f_try1}) in order to take advantage of number-theoretic properties of exponential sums $\sum_{M}e(M \cdot x + |M|^2 t)$ as $M$ ranges over a finite set of integral tuples.

\section{Overview of our goals: arithmetic behavior}\label{sec_overview}
In this section, we define our choice of the function $f$ according to generic parameters and give an overview of the arithmetic we will exploit. 
We will write $x = (x_1,\ldots,x_n) = (x_1,x')$ and set $\Phi_{n-1}(x') = \prod_{j=2}^n \phi(x_j)$. 
We now define $f=f_R$ by
\beq\label{f_try2}
 f(x) = \phi(S_1 x_1) e(Rx_1)  \Phi_{n-1}(x') \sum_{\bstack{m' \in \Z^{n-1}}{R/L \leq m_j < 2R/L}} e(Lm' \cdot x').
\eeq
Here $1 \leq L \leq R$ is an unspecified parameter, which we will choose later; notice that each coordinate of $Lm'$ satisfies $R \leq Lm_j < 2R$, so each $Lm_j$ can play the role of $M_j$ in the discussion of the previous section.
This choice of $f$ has Fourier transform contained in $B_1(R,S_1) \times [R-1,2R+1]^{n-1}$.
 As mentioned above, there exists  $R_1=R_1(n,\sig)$ such that  this will lie in the  annulus $A_n(R,4\sqrt{n})$ for all $R \geq R_1$, since
\beq\label{sig}
S_1 = R^\sig \quad \text{for some $0 \leq \sig \leq 1/2$}.
\eeq

We compute that for $f$ as defined above,
\begin{multline}\label{op_f}
 ( e^{it \Del} f)(x) =\frac{1}{2\pi}\int_{\R} \hat{\phi}(\lam)e((R + \lam S_1)x_1 + (R + \lam S_1)^2 t) d\lam 
	\\\times \frac{1}{(2\pi)^{n-1}} \int_{\R^{n-1}}  \hat{\Phi}_{n-1}(\xi')  \sum_{\bstack{m' \in \Z^{n-1}}{R/L \leq m_j < 2R/L}}  e( (\xi' + Lm')\cdot x'
	  + | \xi' + Lm' |^2 t)  d \xi'.
	  \end{multline}
We now give an overview of how we will show that this is large, in the sense of (\ref{Bou''}).
Define for each $u  \leq  2R/L$,
\beq\label{S_dfn}
S(x',t; u):= \sum_{\bstack{m' \in \Z^{n-1}}{R/L \leq m_j < u}}    e(Lm'\cdot x' + L^2|m' |^2 t).
\eeq	  
Motivated by Section \ref{sec_basic}, we will focus on a set of $x$ such that for each $x$ there are values of $t$ for which we can perform an approximation argument to remove the quadratic behavior in $\lam$ and $\xi'$ in (\ref{op_f}), and then use the fact that $\phi(0)=1$ in order to show that, up to certain error terms, (\ref{op_f}) is 	 controlled by $S(x',t;2R/L)$.
Our goal is then to estimate the magnitude of $S(x',t;2R/L)$ from below, and the magnitude of the error terms from above.
We recall that each integral and sum will factor into a 1-dimensional version. When bounding the error terms, it is useful to define 
\beq\label{Wt_dfn}
 W(t)  := \sup_{v \in [0,2\pi]} \left|\sum_{R/L \leq m < 2R/L}  e(vm + L^2m^2 t)   \right| ,
 \eeq
where $W$ stands for ``Weyl sum.''
In order to understand what a satisfactory upper bound for $W(t)$ will be, we first need to gain  an understanding of a lower bound for $S(x',t;2R/L)$. Here we will need to understand how $x'$ and $t$ are approximated by rationals, and then we will aim to reduce to a ``complete exponential sum,'' which we can evaluate precisely. In order to orient ourselves, we now review the key arithmetic facts that  underpin the entire argument, before turning to a rigorous analysis in Section \ref{sec_reduce}.

So far we have restricted $x$ to a small neighborhood in $[-c_1,c_1]^n$ and chosen $t$ to lie in a certain neighborhood as in (\ref{t_choice}), with the remaining flexibility to choose $\tau$.  
We may further regard $x$ modulo $2\pi$, so that upon rescaling and defining
\beq\label{xy}
 s := L^2 \tau, \qquad y_1 := -\frac{L^2}{2R}x_1 \modd{2\pi} , \qquad y_j := Lx_j \modd{2\pi}, j=2,\ldots, n
 \eeq
we have $y \in [0,2\pi]^n \simeq \mathbb{T}^n$   and we may write $S(x',t;2R/L)$ as a product over $j=2,\ldots, n$ of the 1-dimensional sums
 \beq\label{sum_j} 
\sum_{R/L \leq m_j < 2R/L} e( m_j y_j + m_j^2(y_1 + s) ) .
\eeq
Here we note that the highest-order coefficient (and thus the most interesting) is $y_1 +s$.
 
Now we will further restrict our choice of $x$ by restricting $y$ to a certain set $\Omega \subset \mathbb{T}^n$, which we will later define precisely by taking appropriately small neighborhoods around a collection of rational points (scaled by $2\pi$). Suppose for the moment that $y_1$ is well-approximated by $2\pi a_1/q$ and $y'$ is well-approximated by $2\pi a'/q,$ where $a'/q= (a_2/q,\ldots, a_n/q)$. 
Here it is natural to assume that 
\beq\label{RL_assumption}
R/L \geq q,
\eeq
as we will later ensure through our choice of $L$, so that the range of summation in (\ref{sum_j}) contains at least one complete set of residues modulo $q$.
Given $x$ (or correspondingly $y$), we will then  choose $t$ (and thereby $\tau$ and its corresponding $s$) so that 
\beq\label{y_choice}
y_1 +s= 2\pi \frac{a_1}{q} .
\eeq
 Then we will replace $y'$ by $2\pi a'/q$ by an approximation argument, so that we may shift our attention (up to an error we will show is acceptable)  from (\ref{sum_j}) to the sum
\beq\label{sum_j_a/q} 
\sum_{R/L \leq m_j < 2R/L} e( 2\pi m_j \frac{a_j}{q} +  2\pi m_j^2 \frac{a_1}{q}) .
\eeq
In order to provide a lower bound for this sum, we will break it into complete quadratic Gauss sums (up to an acceptable error). For any $a,b \in \Z$ we define the Gauss sum
\beq\label{sum_j_a/q_q} 
G(a,b;q)=\sum_{m \modd{q}} e( 2\pi m \frac{b}{q} +  2\pi m^2 \frac{a}{q} ).
\eeq
We can   evaluate this complete exponential sum precisely:
\begin{lemma}[Gauss sum]\label{lemma_Gauss}
For any  $a \in \Z$ with $(a,q)=1$ and any $b \in \Z$,
\begin{enumerate}
\item $|G(a,b;q)| = q^{1/2}$, if $q$ is odd,
\item $G(a,b;q)=0$ if $q \con 2 \modd{4}$ and $b$ is even, or $q \con 0 \modd{4}$ and $b$ is odd,
\item $|G(a,b; q)| = (2q)^{1/2}$, if $q \con 2 \modd{4}$ and $b$ is odd, or $q \con 0 \modd{4}$ and $b$ is even.
\end{enumerate}
\end{lemma}
We provide a proof of this classical fact  in Appendix B.
We see that (\ref{sum_j_a/q}) is a sum of $\left\lfloor R/(Lq) \right\rfloor$ copies of $G(a_1,a_j;q)$, plus a possible incomplete sum of length $<q$, of the form
\beq\label{inc_Gauss}
\tilde{G}(u,u') := \sum_{u \leq m_j \leq u'} e( 2\pi m_j \frac{a_j}{q} +  2\pi m_j^2 \frac{a_1}{q} ),
\eeq
for some $1 \leq u \leq u'  \leq q$ with $u'-u<q$. 
Hence (at least in the nonzero cases of Lemma \ref{lemma_Gauss}),  the sum (\ref{sum_j_a/q}) is proportional in absolute value to
\beq\label{main_outcome}
 \left\lfloor  \frac{R}{Lq}  \right\rfloor q^{1/2} + E_j,
\eeq
in which 
\[ |E_j| \leq  \sup_{\bstack{1 \leq u \leq u' \leq q}{u'-u<q}} |\tilde{G}(u,u')|.\]

To bound $|\tilde{G}(u,u')|$ from above, we will apply another classical result, the quadratic case of the Weyl bound:

\begin{lemma}[Weyl bound]\label{lemma_Weyl}
 Suppose that $f(x) = \al x^2 + \be x$ is a real-valued polynomial with $\al$ such that 
\[ \left|\al - \frac{a}{q}\right| \leq \frac{1}{q^2},\]
where $(a,q)=1$. Then there exists a constant $C_0$ independent of $f,a,q,M,N$ such that 
\[\left| \sum_{M \leq n < M+N} e^{2\pi i f(n)} \right|\leq C_0 \left(\frac{N}{q^{1/2}}  + q^{1/2}\right) (\log q)^{1/2}.\]

\end{lemma}
We   provide a proof of this classical  fact  in Appendix B.
In particular,  Lemma \ref{lemma_Weyl} shows that 
\[ \sup_{1 \leq u  \leq u' \leq q} |\tilde{G}(u,u')| \leq 2C_0q^{1/2} (\log q)^{1/2}.\]
Thus as long as we choose $L,q$ such that $  \left\lfloor R/(Lq) \right\rfloor$ is sufficiently large relative to $2C_0(\log q)^{1/2}$ for all sufficiently large $q$, say 
\beq\label{delta_motivation}
R/L \geq q^{1+ \Del_0}
\eeq
for some $\Del_0>0$,
the main term in   (\ref{main_outcome}) will dominate the error term, and
will provide a lower bound  that  is proportionate in absolute value to
\beq\label{main_term}
\frac{R}{Lq^{1/2}}.
\eeq
Since $S(x',t;2R/L)$ is a product of $n-1$ sums of the form (\ref{sum_j}), we thus expect this procedure will produce
a lower bound for $|S(x',t;2R/L)|$ that is proportionate to
\beq\label{main_term_prod}
\left( \frac{R}{Lq^{1/2}} \right)^{n-1}.
\eeq

Importantly, once we have this goal in mind, it establishes an acceptable upper bound  for all the error terms we encounter in approximation arguments. We will also find the Weyl bound of Lemma \ref{lemma_Weyl} useful in bounding $W(t)$ from above.
In particular, by the definition of $t, \tau, s, y_1$,  $W(t)$ as defined in (\ref{Wt_dfn}) can also be written as
\beq\label{W_dfn2}
 W(t)  = \sup_{v \in [0,2\pi]} \left|\sum_{R/L \leq m < 2R/L}  e(vm + m^2 (y_1 +s))   \right| .\eeq
Note that the Weyl bound is uniform in the linear coefficient of the phase polynomial. 
Recalling from the above sketch that given $x$ (or correspondingly $y$), we will then  choose $t$ (and thereby $\tau$) so that (\ref{y_choice}) holds, 
we may apply the Weyl bound to see that 
\beq\label{W_bound}
 W(t) \leq  C_0\left( \frac{R}{Lq^{1/2}} + q^{1/2}\right) (\log q)^{1/2} \leq 2C_0 \frac{R}{Lq^{1/2}} (\log q)^{1/2},
 \eeq
where we have in the last inequality applied our  assumption (\ref{RL_assumption}). 
This bound for $W(t)$ is roughly comparable in size to the main term in (\ref{main_term}). At first glance this appears dissatisfying, since we need the main term in (\ref{main_outcome}) to be an order of magnitude larger than all error terms. But the crucial fact is that $W(t)$ will   appear accompanied by a factor $|t|$ (due to differentiation occurring in integration by parts). The small magnitude of $|t|$ will play a critical role, in combination with (\ref{W_bound}), to control  error terms.

 \subsection{Computing the $L^2$ norm $\|f\|_{L^2}$}
We conclude this section with the simple computation of the $L^2$ norm of $f$, which we will use as a normalizing factor in the inequality (\ref{Bou''}).
We recall the definition of $f$ in (\ref{f_try2}); by Plancherel's theorem, it is equivalent to compute $\|\hat{f}\|_{L^2}$, 
where 
\[ \hat{f}(\xi_1,\xi') = \sum_{\bstack{m' \in \Z^{n-1}}{R/L \leq m_j < 2R/L}} g_{m'}(\xi_1,\xi'),\]
in which 
\[ g_{m'}(\xi_1,\xi') = \frac{1}{S_1} \hat{\phi} \left(\frac{\xi_1 - R}{S_1}\right)  \hat{\Phi}_{n-1}(\xi' - Lm') .
\]
If we let $\Bcal$ denote the box $[-S_1,S_1] \times [-1,1]^{n-1}$, then $g_{m'}$ is supported in 
the shifted box $\Bcal  + (R,Lm')$. In particular, as long as $L \geq 4$, say (which we will later ensure), as $m'$ varies over tuples in $\Z^{n-1}$, any two distinct tuples $m' \neq m''$ have the property that the supports of $g_{m'}$ and $g_{m''}$ are disjoint. Thus
\[ \|\hat{f} \|_{L^2}^2  =  \sum_{\bstack{m' \in \Z^{n-1}}{R/L \leq m_j < 2R/L}} \| g_{m'} \|^2_{L^2}.\]
Thus upon computing that 
$ \| g_{m'} \|_{L^2} = S_1^{-1/2} \| \hat{\Phi}_n \|_{L^2} =  (2\pi)^{n/2} S_1^{-1/2} \|\phi\|_{L^2}^n$, we see that 
\beq\label{f_norm_computation}
\|f\|_{L^2}= (2\pi)^{-n/2} \|\hat{f} \|_{L^2} =   S_1^{-1/2} (R/L)^{\frac{n-1}{2}}  \| \phi \|^n_{L^2}.
 \eeq
 We will use this in our final verification of (\ref{Bou''}).

\subsection{Organization of the rigorous argument}
Having sketched an overview of our plan, we now carry it out rigorously. 
In Section \ref{sec_reduce} we show how to pass from $(e^{it\Del }f)(x)$   to the sum $S(x',t;2R/L)$, up to certain error terms. 
In Section \ref{sec_omega} we define the sets $\Omega$ and $\Omega^*$ that allow us to exploit arithmetic in $S(x',t;2R/L)$, and we compute the measure of these sets. 
In Section  \ref{sec_evaluate} we evaluate $S(x',t;2R/L)$ to compute a main term.
In Section \ref{sec_parameters} we bound all  the error terms accumulated  and assemble all the assumptions we have made so far about the relationships of the parameters. We then make optimal parameter choices and complete the proof of Bourgain's criterion, in the form of (\ref{Bou''}).

\section{Reducing to arithmetic behavior}\label{sec_reduce}

In this section, we show that for $f$ defined in (\ref{f_try2}), in the neighborhood of $x$ we consider, and for   $t$  satisfying the requirements of (\ref{t_choice}) and (\ref{t_fact1}),
\beq\label{conseq1}
 |(e^{it \Del} f)(x) |   \geq (1-c_0)^n \left|  S(x',t;2R/L) \right| -( |E(1)| + |E(2)|) 
\eeq
in which the error terms satisfy upper bounds given in (\ref{E1_bound}) and (\ref{E2_bound}), respectively. This makes the ideas outlined in Section \ref{sec_overview} rigorous.

 At this point we note that we may start with a choice of $c_0$ as small as we like, and while this determines an upper bound on $\del_0 = \del_0(c_0)$, we may also choose $\del_0$ smaller if we wish. Thus for now we suppose that 
 \beq\label{c_del}
  c_0 \leq c_0^* = c_0^*(n,\phi), \qquad   \del_0 \leq \del_0^* = \del_0^*(n,\phi), 
   \eeq
 and at the end of the paper we will see what to impose as upper bounds on $c_0^*, \del_0^*$, depending only on $n, \phi.$

\subsection{The integral over $\lam$}
We first show that in absolute value, the contribution of the integral over $\lam$ in (\ref{op_f})  has magnitude at least $1 - c_0$.
By definition, this contribution is equal to 
\[ e(Rx_1 + R^2 t)\cdot \frac{1}{2\pi}  \int_{\R} \hat{\phi}(\lam) e(\lam( S_1 x_1 + 2 RS_1 t))   e(S_1^2 t\lam^2) d \lam . \]
By Lemma \ref{lemma_int_approx}, this expression is equal to 
\[
 e(Rx_1 + R^2 t) e(S_1^2t) \phi(S_1(x_1 + 2Rt))  + E_1 
\]
in which $|E_1| \leq c_0/2$ by the property (\ref{t_fact1}) of $t$.
Furthermore  by the choice of $t$ in (\ref{t_choice}) and (\ref{t_fact1}) we know that   
\[\phi(S_1(x_1 + 2Rt)) =1 + E_1'\] with $|E_1'| \leq c_0/2$.
Thus in (\ref{op_f}) the integral over $\lam$ is equal to 
\beq\label{lam_end_goal}
 e(Rx_1 + R^2 t) e(S_1^2t)  + E_1'',
 \eeq with 
  $|E_1''| \leq c_0$. This proves our claim.

\subsection{The integral over $\xi'$}
We now show that the integral over $\xi'$ in (\ref{op_f}) evaluates to $S(x',t;2R/L)$, up to  error terms $E(1)$ and $E(2)$.
The integral over $\xi'$ is equal to 
\beq\label{int_xi'}
\frac{1}{(2\pi)^{n-1}}   \int_{\R^{n-1}}  \hat{\Phi}_{n-1}(\xi')  \sum_{\bstack{m' \in \Z^{n-1}}{R/L \leq m_j < 2R/L}}  e(Lm'\cdot x' + L^2|m' |^2 t) e(  \xi' \cdot (x'+2Lm't )) 
	  e( | \xi' |^2   t)  d \xi'.
	  \eeq
The key step is to show that this is equal to 
\beq\label{step1_xi'}
 e(t)^{n-1}  \Phi_{n-1}(x' + (R',\ldots, R')t) S(x',t;2R/L) + E(1) + E(2)
 \eeq
	 in which $|E(1)|$ and $|E(2)|$ are bounded by (\ref{E1_bound}) and (\ref{E2_bound}), respectively.
	 Here we have defined $R' = 2L ( \lceil 2R/L \rceil -1)$. (This notation will only be relevant for this section, and the only fact we will use about it is that $R' \leq 4R$.)
Once we have shown this, we simply note that by our choice of $t$ we have $|t| \leq \del_0/(8R)$ and so certainly $R'|t| \leq \del_0/2$; hence for each $j$ we have 	$\phi(x_j + R' t)  \geq 1 - c_0/2 \geq 1-c_0$, and hence $|\Phi_{n-1}(x' + (R',\ldots, R')t)| \geq (1-c_0)^{n-1}$.
	Assembling this result for the integral over $\xi'$ in (\ref{op_f}) with the result (\ref{lam_end_goal}) for the integral over $\lam$, we can conclude that  (\ref{conseq1}) holds, as soon as we have proved (\ref{step1_xi'}).

Our first step in proving (\ref{step1_xi'}) is to approximate (\ref{int_xi'}) so as to remove the factor $e( | \xi' |^2   t)$, and then we can use Fourier inversion to reveal   $ \Phi_{n-1}(x' + 2Lm't)$. The second step is to pull this factor out of the sum over $m'$ by a second approximation argument, thus isolating the exponential sum $S(x',t;2R/L)$.

\subsubsection{Removal of the quadratic phase}
 In the first step, we apply  Lemma \ref{lemma_int_approx} to remove the quadratic factor $e(|\xi'|^2t)$, in order to show 
 that (\ref{int_xi'}) is equal to 
 \beq\label{int_xi'_2}
  e(t)^{n-1}    \sum_{\bstack{m' \in \Z^{n-1}}{R/L \leq m_j < 2R/L}}  \Phi_{n-1}(x' + 2Lm't)  e(Lm'\cdot x' + L^2|m' |^2 t) + E(1).
  \eeq
To carry this out, we factor  (\ref{int_xi'}), recalling that $\hat{\phi}$ is supported in $[-1,1]$, and apply Lemma \ref{lemma_int_approx} to the $\xi_j$-th integral, with $h(\xi_j) = \xi_j^2 t$ so that $\|h'\|_{L^\infty[-1,1]} \leq 2|t|$. We obtain
\begin{multline}\label{j_int}
 e(  t) \frac{1}{2\pi} \int_{-1}^1  \hat{\phi} (\xi_j)  \sum_{R/L \leq m_j < 2R/L}  e(Lm_j x_j + L^2m_j^2 t) e(  \xi_j ( x_j + 2Lm_jt))
	   d \xi_j + E_j
	  \\ 
	  = e(  t)    \sum_{R/L \leq m_j < 2R/L}  \phi(x_j +2Lm_jt) e(Lm_j x_j + L^2m_j^2 t) + E_j,
	  \end{multline}
with 
$
 |E_j| \leq (4/2\pi) |t| \| \mu \|_{L^1[-1,1]} \leq  |t| \| \mu \|_{L^1[-1,1]},$ in 
 which $\mu$ is the integrand in (\ref{j_int}). Using the function 
  $W(t)$ as defined in (\ref{Wt_dfn}), we see that 
  \beq\label{Ej_est}
 |E_j|    \leq   |t|   W(t)\| \hat{\phi} \|_{L^1}.
	  \eeq
Of course the main term  on the left-hand side of (\ref{j_int}) can be   bounded above by 
\beq\label{full_int_est}
 W(t)  \|\hat{\phi}\|_{L^1}.
 \eeq
 Consequently, when we multiply together the expressions (\ref{j_int}) for $j=2,\ldots, n$, we see that the full integral over $\xi'$ given in (\ref{int_xi'}) is equal to (\ref{int_xi'_2}), in which 
the error term $E(1)$  is the sum of all possible cross terms, as $\ell$ varies from $0$ to $n-2$, with $\ell$ factors bounded by (\ref{full_int_est}) and the remaining $(n-1-\ell)$ factors of the form $E_j$ and bounded by (\ref{Ej_est}). 
The largest such terms occur for $\ell=n-2$, when there is only one factor of the small term $|t|$.	Thus we record the bound
	\beq\label{E1_bound}
	|E(1)| \leq C_1 \|\hat{\phi}\|_{L^1}^{n-1} W(t)^{n-1} |t|,
	\eeq
	for a constant $C_1 = C_1(n)$.
	We will later show that for appropriate choices of $x,t$, since $t$ is chosen to be small as in (\ref{t_choice}), $|E(1)|$ will be sufficiently small relative to $|S(x',t;2R/L)|$.

\subsubsection{Isolation of the exponential sum}
We turn our focus to the sum in (\ref{int_xi'_2}). We assume that $x',t$ are fixed. We would like to approximate $\Phi_{n-1}(x' + 2Lm't)$ by $1$, but we cannot do this uniformly in $m'$, and thus we must first remove the factor $\Phi_{n-1}(x' + 2Lm't)$ from the sum over $m'$.
We again work one dimension at a time. Our tool is partial summation, which shows that a sum of complex numbers $a_n$ weighted by a $C^1$ weight $h(n)$ can be well-approximated by the sum of $a_n$ alone, as long as the derivative of $h$ is sufficiently small. 
\begin{lemma}[Partial summation]\label{lemma_partial_sum}
Suppose $a_n$ is a sequence of complex numbers and $h$ is a $C^1$ function on $\R$. Upon setting $A(u) = \sum_{M\leq n \leq u} a_n$, then
\[ \sum_{n=M}^{M+N}a_n h(n)  = A(M+N)h(M+N) - \int_M^{M+N} A(u) h'(u) du.\]
\end{lemma}
\begin{proof}
It suffices to observe that
\[\sum_{n=M}^{M+N} a_n ( h(M+N) - h(n)) = \sum_{n=M}^{M+N} a_n \int_{n}^{M+N} h'(u) du = \int_{M}^{M+N} (\sum_{n=M}^u a_n )h'(u)du.\]
\end{proof}
For each $j=2,\ldots, n$, define $S_j(u)$ for any $R/L \leq u \leq 2R/L$  by
\beq\label{Wt_dfn_tilde}
S_j(u) := \sum_{R/L \leq m_j < u}   e(L m_j x_j+ L^2m_j^2 t) .
\eeq
Note that this depends on $x_j,t$ as well, which are fixed for the present discussion.
Apply  Lemma \ref{lemma_partial_sum} to the $m_j$-coordinate sum that is a factor in (\ref{int_xi'_2}) to see that 
\beq\label{sum_j_partial}
 \sum_{R/L \leq m_j < 2R/L} \phi (x_j + 2Lm_jt)  e(Lm_j x_j + L^2m_j^2 t)
 	\\
	 = \phi (x_j + R't) S_j(2R/L)   + E_j(2),
\eeq
in which we recall the notation $R' = 2L ( \lceil 2R/L \rceil -1)$, and the error term is
\[ E_j(2)   =- \int_{R/L}^{2R/L} \left( \sum_{R/L \leq m_j < u}   e(Lm_j  x_j + L^2m_j^2 t) \right) (2Lt )\phi'(x_j + 2Lut) du.\]
We may bound $|E_j(2)|$ by
\beq\label{Ej_est'}
 |E_j(2)| \leq  (R/L) 2L |t|   \| \phi' \|_{L^\infty} \sup_{R/L \leq u \leq 2R/L} |S_j(u)|.
 \eeq
We also note that the main term on the right-hand side of (\ref{sum_j_partial}) can be bounded by 
\beq\label{full_int_est'}
\| \phi \|_{L^\infty} |S_j(2R/L)|.
\eeq
We now multiply together the expressions (\ref{sum_j_partial}) for $j=2,\ldots, n$ to see that 
\[ 
 \sum_{\bstack{m' \in \Z^{n-1}}{R/L \leq m_j < 2R/L}}  \Phi_{n-1}(x' + 2Lm't)  e(Lm'\cdot x' + L^2|m' |^2 t) 
 	 =   \Phi_{n-1}(x' + (R',\ldots, R')t) S(x',t;2R/L) + E(2)
\]
in which $E(2)$ results from all possible cross terms, as $\ell$ varies from 0 to $n-2$, with $\ell$ factors bounded by (\ref{full_int_est'}), and the remaining $(n-1-\ell)$ factors of the form $E_j(2)$ and bounded by (\ref{Ej_est'}). 
Precisely, 
\beq\label{E2_bound}
|E(2)| \leq \sum_{\ell=0}^{n-2} C_\ell \left( \| \phi \|_{L^\infty} \cdot \sup_{2 \leq j \leq n}|S_j(2R/L)| \right)^\ell \left(2R|t|   \| \phi' \|_{L^\infty} \sup_{2 \leq j \leq n} \sup_{R/L \leq u \leq 2R/L} |S_j(u)| \right)^{n-1-\ell},
\eeq
for some combinatorial constants $C_\ell$.
This proves the claim (\ref{conseq1}), and completes the technical work of this section.

\begin{remark}
Recalling   the discussion of the previous section, we may anticipate that for each $x'$ we consider, we will choose $t$ appropriately so that $|S_j(u)|$ is proportional to $\lfloor (u-R/L)/q \rfloor q^{1/2}$ (up to an error term of size $2C_0 q^{1/2} (\log q)^{1/2}$) for all $u \leq 2R/L$ and for each $2 \leq j \leq n$.   
Later, we will choose $q$ to lie in a range $4\mu_0 Q \leq q \leq 4Q$ for a  constant $0<\mu_0<1$, and a parameter $Q$ that is a small power of $R$, to be chosen at the end of the argument. This, combined with the assumption that $R/L\geq Q^{1+ \Del_0}$ for a small parameter $\Del_0$ to be chosen later,   will allow us in (\ref{crude_Sju}) to bound the contribution of $|S_j(u)|$ by at most a multiple of $R/(LQ^{1/2})$, uniformly for $R/L \leq u \leq 2R/L$. 

Once we have verified this,  
the largest contribution to $E(2)$ comes from the term $\ell=n-2$, leading to a bound of the form
\beq\label{E2_bound_wish}
 |E(2)| \leq C_2 R |t|  (\|\phi\|_{L^\infty} + \| \phi' \|_{L^\infty})^{n-1}\left( \frac{R}{LQ^{1/2}} \right)^{n-1}
\eeq
for some other constant $C_2$ depending on $n, \Del_0$. Such an upper bound will be sufficient, relative to the main term in (\ref{conseq1}) (proportional to (\ref{main_term_prod})), due to the presence of the factor $R|t| \leq \del_0/8 $ (see (\ref{t_choice})),  as long as we take $\del_0$ to be sufficiently small relative to $c_0,C_2,n, \|\phi\|_{L^\infty}, \| \phi' \|_{L^\infty}$.
Since we cannot prove (\ref{E2_bound_wish}) rigorously  until we have chosen the set of $x,t$ we consider, for the moment we record (\ref{E2_bound}) as our upper bound for $|E(2)|$, and return to prove (\ref{E2_bound_wish}) later.
\end{remark}

\section{Construction of the sets $\Omega$ and $\Omega^*$}\label{sec_omega}
Our starting point in this section is the key result (\ref{conseq1}) of the previous section.
So far we have restricted to a small neighborhood of $x$ in  $[-c_1,c_1]^n$, and we have chosen $t$ and accordingly $\tau$ so that (\ref{t_choice}) and (\ref{t_fact1}) hold.
From these, we correspondingly define the variables $s, y_1,y'$  as in (\ref{xy}).
Our goal in this section is to construct a set $\Omega$, comprised of small neighborhoods of $2\pi a_j/q$ for certain rationals $a_j/q$ with $q$ of about size $Q$, for a parameter $Q$ to be chosen later in terms of $R$. This set $\Omega$ will have the property  that for any $x$ such that the corresponding $y$ lies in $\Omega$, we can choose $t$ so that the behavior of $S(x',t;2R/L)$ is dominated by Gauss sums, which we then evaluate precisely in Section \ref{sec_evaluate}. 

It is natural to assume that we choose $Q$ such that 
\beq\label{RLQ}
 \frac{R}{L} \geq Q \eeq
so that for each integer $q$ of about size $Q$, any $R/L$ consecutive integers contain at least $q$ consecutive integers. 
But in fact we recall from  our motivating discussion in (\ref{delta_motivation}) (which we will make precise momentarily) that we need $R/L$ to be a bit larger, and thus we now formally assume that for some small $0 < \Del_0 \leq 1$ to be chosen later, 
\beq\label{RLQ_try2}
 \frac{R}{L} \geq Q^{1+\Del_0}.
\eeq 
 We also make the weak assumption that $Q$ grows like some power of $R$ (so in particular, for all $R$ sufficiently large, $Q$ is at least as large as any absolute constant, such as $(4\pi)^n$). We will write this as an assumption that $L=o(R)$ and there exists some $\ep_1>0$ such that 
 \beq\label{RLQ_try3}
 Q \geq \left(\frac{R}{L} \right)^{\ep_1}.  
\eeq
 (In the discussion below, we can proceed from first principles with the weaker assumption (\ref{RLQ}) until equation (\ref{enlarge_E4}) below, at which point the stronger assumption (\ref{RLQ_try2}) allows us to consolidate error terms into what we call $E(3)$ below.) 
 
 \subsection{The three key properties of the set $\Omega$}\label{sec_3_ppties}
We now state the key properties of the set $\Omega$ in the form of a claim with three parts. The motivation for the assumptions in (\ref{UV_assps}) on the relative sizes of $L, R, S_1, Q$ will become clear momentarily. 

Fix any $Q \geq 1$ satisfying (\ref{RLQ_try2}) and (\ref{RLQ_try3}) for $\Del_0$ and $\ep_1$. Assume $L =o(R)$ and $R = o(L^2)$. Fix $\mu_0 = (4\pi)^{-n}$. 
Assume
 \beq\label{UV_assps}
 \frac{1}{Q} \leq \frac{L^2}{S_1R}, \qquad \frac{\pi}{(\mu_0Q)Q^{1/(n-1)}} \leq C_3'\left(\frac{R}{L} \right)^{-1},
 \eeq
 for some $C_3' = C_3'(n)$.
 Fix
  any small absolute constants $c_3 \leq \min\{c_2, 1/2\pi\}$, $c_4< 1/2$. Then there exists a set $\Omega \subset \mathbb{T}^n\simeq [0,2\pi]^n$, and a set  $\Omega^* \subset [-c_1,-c_1/2] \times [-c_1,c_1]^{n-1}$, such that for each $x \in \Omega^*$  the corresponding $y  = (y_1,y')$ as defined by (\ref{xy}) belongs to $\Omega$, and such that the sets $\Omega$ and $\Omega^*$ defined using $c_3,c_4$ have the following properties.

{\bf Property (I):} For  every  $x=(x_1,x') \in \Omega^*,$   there exists a $t \in (0,1)$ satisfying the conditions  (\ref{t_choice}) and (\ref{t_fact1}) and an integer $q \in [4\mu_0Q, 4Q]$ such that  
\beq\label{S_E3}
 |S(x',t;2R/L)| = \left( \frac{\sqrt{2}R}{Lq^{1/2}}\right)^{n-1}    + E(3)  ,
 \eeq
in which
\beq\label{E3_E4}
|E(3)| \leq  C_3 \left(  \frac{R}{LQ^{1/2}} \right)^{n-1} (c_4 + (R/L)^{-\Del_0\ep_1 /2})   
\eeq
for some $C_3 = C_3(n, \Del_0,\mu_0)$. 

{\bf Property (II):}  The measure of $\Omega$ satisfies the property that for any $\ep_0>0$, there exists a constant $0<c_{\ep_0}<1$ such that
 \beq\label{Omega_size}
 |\Omega| \geq  c_{\ep_0} c_3 c_4^{n-1}3^{-(n-1)}2^{-n}  \mu_0  Q^{-\ep_0} .
 \eeq

{\bf Property (III):}  In measure $|\Omega^*| \geq c_1' |\Omega|$, with a constant $c_1'$ depending only on $c_1,n$.

We now carefully motivate how one would construct $\Omega$ from first principles, before formally defining a more complicated, rigorous, version, which allows us to prove that (I), (II), and (III) hold. (Other constructions are possible, but we tried to use an intuitive approach here.)

\subsection{First informal model for the set $\Omega$}
We now define a first guess for the set $\Omega \subset \mathbb{T}^n$ according to small constants $c_3,c_4$ and small parameters $U,V$, which we will choose momentarily in terms of $Q$. 
We consider a model for $\Omega$ defined by
 \[  \Union_{\bstack{\mu_0 Q \leq q \leq Q}{a_1,a'}} \{ (y_1,y') \in \mathbb{T}^n: |y_1 - \frac{2\pi a_1}{q}| < c_3 U, |y_j - \frac{2\pi a_j}{q}| < c_4V, \; j=2,\ldots,n\} \]
 in which the union denotes that $q$ runs over integers in the range $[\mu_0 Q,Q]$, $a_1$ runs over the residues $1 \leq a_1 \leq q$ with $(a_1,q)=1$, and $a' = (a_2,\ldots, a_n) \in \Z^{n-1}$ runs over all residues, $1 \leq a_j \leq q$. 
 We will later modify this model into a formal definition of $\Omega$, after determining  appropriate choices of $U,V$.
 
A reasonable initial hope is to choose $U,V$ so that $|\Omega|$ is at least a positive constant, independent of $R$; this encourages us to choose $U,V$ large.    On the other hand, we need $U,V$ to be small enough that the approximations of the $y_j$ are sufficiently accurate for partial summation to succeed in passing from (\ref{sum_j}) to (\ref{sum_j_a/q}) without accumulating large errors. Note that no advantage is gained by taking $U,V$ any larger than $Q^{-1}$, since we are approximating by denominators of size approximately $Q$.

\subsection{Choosing $t$ to avoid approximations in the quadratic term: upper bound for $U$}
 In order to determine how we must reasonably choose $U$, we recall that for each fixed $x \in \Omega^*$,  we are allowed to choose  $t = -x_1/(2R) + \tau$ for any  $|\tau| \leq c_2/(S_1R)$; i.e. by (\ref{xy}) we are allowed to choose any $s = L^2 \tau$ with $|s| \leq c_2 L^2/(S_1R)$. This motivates us to require $c_3 \leq c_2$ and 
\[
U \leq \frac{L^2}{S_1R}.
\]
With these choice for $c_3$ and $U$, the set $\Omega$ has the following property: given any $x \in \Omega^*$ and the corresponding $y_1$ in an interval centered at $2\pi a_1/q$, there exists $s$  such that 
\beq\label{ys_choice}
y_1 + s = 2\pi a_1/q,
\eeq
 and $|s| \leq c_2L^2/(S_1R)$. Upon choosing this $s$ the corresponding $t,\tau$ satisfy the usual requirement (\ref{t_choice}).
Conveniently, this ability to \emph{choose} $s$ (or equivalently, to choose $\tau$) avoids an approximation  to obtain a rational coefficient for the quadratic term in the exponential sums.

\subsection{Approximations in the linear term: upper bound for  $V$}
In contrast, for $j=2,\ldots, n$, to pass from $y_j$ to $2\pi a_j/q$ inside the linear term in the exponential sum (\ref{sum_j}), we will require an approximation lemma, which will force an upper bound on $V$.
 Given a real-valued function $f$, we temporarily  use the notation that 
\[ S(f;M,N) = \sum_{n=M}^{M+N} e(f(n)).\]
The content of the following lemma is that given a real-valued function $h$, $|S(f+h;M,N)|$ is   proportional to $|S(f;M,N)|$  if the derivative of $h$ is sufficiently small.
\begin{lemma}[Partial summation II]\label{lemma_sum_approx}
Let $f,h$ be real-valued functions of $\R$ and in addition assume that $h$ is $C^1$.
Then 
\[ S(f+h;M,N) = S(f;M,N)e(h(M+N)) + E\]
where 
\[ |E| \leq  \sup_{u \in [0,N]} |S(f;M,u)| \cdot \| h' \|_{L^\infty[M,M+N]} \cdot N.\]
\end{lemma}
\begin{proof}
It suffices to observe that by Lemma \ref{lemma_partial_sum},
 \[S(f+h;M,N) = S(f;M,N)e(h(M+N)) - \int_M^{M+N}S(f;M,u-M) h'(u)e(h(u)) du.\]
 \end{proof}
Thus as a general principle, to conclude that $ |S(f+h;M,N)| \geq (1- \al_0) |S(f;M,N)|$ for a certain constant $\al_0<1$, 
 it suffices  to bound $|S(f;M,u)| $ by an increasing function in $u$, so that $\sup_{u \in [0,N]} |S(f;M,u)|\leq |S(f;M,N)|$, and to show $ \| h' \|_{L^\infty[M,M+N]}  \leq \al_0 N^{-1}$.
 
 Recall that $S(x',t; 2R/L)$ is a product of sums of the form $S_j(u)$ defined in (\ref{Wt_dfn_tilde}).
We now record a result for $S_j(u)$ that holds for any $R/L \leq u \leq 2R/L$. This will show us what an acceptable size will be for error terms when replacing $y_j$ by $2\pi a_j/q$, and hence indicate an upper bound on $V$.

 We apply Lemma \ref{lemma_sum_approx} to the sum $S_j(u)$ by setting 
 \[f(m_j) = m_j (2\pi a_j/q) + m_j^2 (y_1 + s), \qquad h(m_j) = m_j(y_j - 2\pi a_j/q).\]
 We also define for any $R/L \leq u \leq 2R/L$, 
 \[ \tilde{S}_j(u): =  \sum_{R/L \leq m_j < u} e( m_j (2\pi a_j/q)+ m_j^2(y_1 + s) ) .\]
 The derivative  satisfies $|h'(m_j)| \leq c_4 V$, and thus Lemma \ref{lemma_sum_approx} shows that  
 for any $R/L \leq u \leq 2R/L$, 
 \beq\label{sum_j_2}
S_j(u)= \tilde{S}_j(u) + E_j(u;3)
  \eeq
 where
\beq\label{sum_est_1_0}
 |E_j(u;3)| \leq c_4   V u  \sup_{0\leq w \leq u} |\tilde{S}_j(w)|.
 \eeq
It remains to bound $|\tilde{S}_j(w)|$ for each $0 \leq w \leq u$, but we will see in Section \ref{sec_main_term1} that it is bounded by an increasing function in $w$, so that 
$\sup_{0 \leq w \leq u} |\tilde{S}_j(w)| \leq |\tilde{S}_j(u)|$. Once we have established this, in order to conclude that  in particular for $u = 2R/L$, $|E_j(2R/L;3)|$ is at most a small positive proportion of the expected main term $|\tilde{S}_j(2R/L)|$, (\ref{sum_est_1_0}) shows us that we must at least ensure that 
\[V \leq C_3'(R/L)^{-1},\]
for some $C_3' = C_3'(n)$. 
Then we can take $c_4$ appropriately small relative to $C_3'$ and other absolute constants.  (This choice of $V$ agrees with $V \leq Q^{-1},$ under (\ref{RLQ}).) This concludes our motivation for upper bounds for $U,V$. We turn to motivations for lower bounds.

\subsection{Conclusions about $U,V$ relative to $Q$}
 
 The computations above indicate different considerations for $U$ and $V$ and thus it is natural to consider the volume of $\Omega$ as a $1 \times (n-1)$-dimensional computation.
In aiming to cover a positive measure of $[0,2\pi]^{n-1}$, it is natural to think of the principle of simultaneous Dirichlet approximation. 
Simultaneous Dirichlet approximation in $n-1$ dimensions   shows that for every $Q \geq 1$, every point $(y_2,\ldots, y_n)$  in $[0,1]^{n-1}$ can 
be approximated by $  (a_2/q,\ldots, a_n/q)$ for some uniform denominator $1 \leq q \leq Q$ with accuracy
\[ |y_j -   a_j/q| \leq \frac{1}{q Q^{1/(n-1)}}, \qquad 2 \leq j \leq n.\]
We provide a proof of this classical result in Appendix B.
In general the lengths of these intervals cannot be shortened by an order of magnitude and still yield boxes that cover $[0,1]^{n-1}$. 
 Thus in order for $\Omega$ to have a chance of covering a positive proportion of $[0,2\pi]^{n-1}$ in its last $n-1$ coordinates,
 we require that $V \geq (q\cdot Q^{1/(n-1)})^{-1}$ for each $q$ we consider. Taking $V$ larger than this will not increase the measure of $\Omega$ by an order of magnitude, and so in our formal definition in the next section, we are motivated to choose $V$  proportional to 
 \[V \approx ((\min q) Q^{1/(n-1)})^{-1}.\] 
(Here, temporarily, we let $\max q$ and $\min q$ denote the maximum and minimum of those denominators we will consider.)

 In order for $\Omega$ to have a chance of covering a positive proportion of $[0,2\pi]$ in its first coordinate, we would then need to 
 have $U$ at least proportional to $q^{-1}$; taking $U$ larger than this would not increase the measure of $\Omega$ by an order of magnitude. We also would like the intervals around $2\pi a_1/q$ to be disjoint at $a_1$ varies,  and thus we set
 \[U = (\max q)^{-1}. 
 \]
 
\subsection{Formal definition of the set $\Omega$}
We now formally define the set $\Omega$. Under the assumptions (\ref{UV_assps}), for small $c_3 \leq \min\{c_2,1/2\pi\}$, $c_4<1/2$ of our choice, and $\mu_0 = (4\pi )^{-n}$, we define $\Omega$ to be the set
\begin{multline*}
  \Union_{\bstack{4\mu_0 Q \leq q \leq 4Q}{ q \con 0 \modd{4}}} \Union_{\bstack{1 \leq a_1 \leq q}{(a_1,q)=1}} \Union_{\bstack{2 \leq a_2, \ldots, a_n \leq 2q}{a_j \con 0 \modd{2}}} \{ (y_1,y') \in [0,2\pi]^n: |y_1 - \frac{2\pi a_1}{q}| < \pi c_3 (4Q)^{-1},  
  \\
   |y_j - \frac{2\pi a_j}{q}| < \pi c_4 ((\mu_0 Q) Q^{1/(n-1)})^{-1}, \; j=2,\ldots,n \}.
\end{multline*}
The conditions that $q \con 0 \modd{4}$ and  $a_j \con 0 \modd{2}$ for $j=2,\ldots, n$ assure that the Gauss sum $G(a_1,a_j;q)$ will fall into  a nonzero  case of Lemma \ref{lemma_Gauss}. 
We now verify properties (II) and (III) according to this definition. We will prove property (I) in Section \ref{sec_evaluate}.
Note that in this construction, $V = \pi((\mu_0 Q) Q^{1/(n-1)})^{-1} \leq C_3'(R/L)^{-1}$ under assumption (\ref{UV_assps}).

 \subsection{Property (II): Volume of $\Omega$}
 
  Property (II) will follow from two facts, which we now prove:
 
 (IIa) Fix any $\ep_0>0$, and $0<c_3< 1/2\pi$. There exists a constant $0<c_{\ep_0}<1$ such that for each   $4\mu_0 Q \leq q \leq 4Q$, as $a_1$ varies over $1 \leq a_1 \leq q$ with $(a_1,q)=1$, the intervals centered at $2\pi a_1/q$ of length $2\pi c_3(4Q)^{-1}$ cover  a subset of $[0,2\pi]$ of measure at least $  c_3c_{\ep_0} \mu_0 Q^{-\ep_0}$. 
 
 (IIb) As long as $\mu_0 \leq (4\pi)^{-n}$, 
 
 \[|  \Union_{\bstack{4\mu_0 Q \leq q \leq 4Q}{ q \con 0 \modd{4}}}  \Union_{\bstack{2 \leq a_2, \ldots, a_n \leq 2q}{a_j \con 0 \modd{2}}}  \{ y' \in \mathbb{T}^{n-1}:   |y_j - \frac{2\pi a_j}{q}| < \pi c_4 ((\mu_0 Q) Q^{1/(n-1)})^{-1} \}| \geq c_4^{n-1}3^{-(n-1)}2^{-n}.\]
 In particular,  (IIb)  informs us which denominators $q$ to pick for the boxes in the last $n-1$ coordinates to cover a positive measure subset of $[0,2\pi]^{n-1}$, and for each such $q$, property (IIa) guarantees a lower bound on the measure covered in the first coordinate. In total, this verifies (\ref{Omega_size}).

\subsubsection{Proof of (IIa)} 
Since $[0,1]$ is covered by intervals centered at $a_1/q$ of length $1/q$ with $1 \leq a_1 \leq q$, $[0,2\pi]$ is covered by intervals centered at $2\pi a_1/q$ of length $2\pi/q$. 
Recall that for any integer $q$, there are $\varphi(q)$ residues relatively prime to $q$, where $\varphi(q) = q\prod_{p|q}(1-1/p)$ is the Euler totient function. In particular,  $(1/2)^{\om(q)} q \leq \varphi(q) \leq q$ for any $q$, where $\om(q)$ denotes the number of distinct prime factors of $q$. There is an absolute constant $c_5$ such that $\om(q) \leq c_5  \log q/ \log \log q$ for all integers $q$ \cite[\S22.10]{HarWri08}. Thus for any $\ep_0>0$ there exists a constant $0<c_{\ep_0}'<1$ such that   $2^{-\om(q)} \geq c_{\ep_0}' q^{-\ep_0}$ for all $q \geq 1$.
Thus for each $4\mu_0 Q\leq q \leq 4Q$, a union of $\phi(q)$ many disjoint intervals of length $2\pi c_3(4Q)^{-1}$ as described in (IIa)  covers a set of measure at least $  c_{\ep_0}'(4Q)^{-\ep_0} (4\mu_0 Q) \cdot 2\pi c_3 (4Q)^{-1} \geq c_{\ep_0} \mu_0 c_3Q^{-\ep_0}$.
 
\subsubsection{Proof of (IIb)}  
This argument uses simultaneous Dirichlet approximation in $n-1$ dimensions, followed by rescaling to ensure the congruence conditions in case (3) of Lemma \ref{lemma_Gauss}. 

The first task is to show we still obtain a positive proportion of the measure if we restrict from $1 \leq q \leq Q$ to a range of $q$ proportional to $Q$.
Let $J(q;a_2,\ldots, a_n)$ denote the product over $j=2,\ldots, n$ of the intervals centered at $ 2\pi a_j/q$ of length $2 \cdot 2\pi (qQ^{1/(n-1)})^{-1}$. By simultaneous Dirichlet approximation in $n-1$ dimensions (rescaled to $[0,2\pi]$), every element in $[0,2\pi]^{n-1}$ lies in at least one such box. Thus
 \[ 
 | \Union_{1 \leq q \leq Q} \Union_{1 \leq a_2, \ldots, a_n \leq q} J(q;a_2,\ldots, a_n) | \geq  (2\pi)^{n-1}\geq 1.
 \]
We next claim that if $\mu_0 \leq (4\pi)^{-n}$ then
 \[ 
|M(\mu_0)| :=  |\Union_{1 \leq q< \mu_0 Q} \Union_{1 \leq a_2, \ldots, a_n \leq q} J(q;a_2,\ldots, a_n) | \leq 1/2.
  \]
We compute an upper bound as follows:
  \begin{align*}
  |M(\mu_0) | &= \int_{[0,1]^{n-1}} \mathbf{1}_{M(\mu_0)} (u) du \\
  	& \leq 	 \int_{[0,1]^{n-1} }
		 \sum_{1 \leq q < \mu_0 Q } \sum_{1 \leq a_2, \ldots, a_n \leq q} \mathbf{1}_{J(q;a_2,\ldots, a_n)}   (u) du	 	\\
	& \leq   \sum_{1 \leq q < \mu_0 Q} \sum_{1 \leq a_2, \ldots, a_n \leq q} (\frac{4\pi}{qQ^{1/(n-1)}})^{n-1}
		\leq (4\pi)^{n-1} Q^{-1} \sum_{1 \leq q < \mu_0 Q}    1
		\\
		& \leq (4\pi)^{n-1} \mu_0  \leq 1/2,
  \end{align*}
  under the assumption $\mu_0 \leq (4\pi)^{-n}$.
Consequently the restricted union of $J(q;a_2,\ldots, a_n)$ over $\mu_0 Q \leq q \leq Q$, $1 \leq a_2, \ldots, a_n \leq q$  has measure at least $1/2$. 

We now rescale each cube by a small constant $0<c_4<1$ of our choice, which we need to ensure certain error terms are small (see (\ref{E3_E4})). Let  $J^*(q;a_2,\ldots, a_n)$ be defined as $J(q;a_2,\ldots, a_n)$ but according to intervals of length $4\pi c_4 (qQ^{1/(n-1)})^{-1}$: thus $J^*(q;a_2,\ldots, a_n)$ is a cube of side-length $4\pi c_4 (qQ^{1/(n-1)})^{-1}$ centered at $(2\pi a_2/q,\ldots, 2\pi a_n/q)$. For each cube, the measure rescales as $|J^*(q;a_2,\ldots, a_n)|=c_4^{n-1}|J(q;a_2,\ldots, a_n)|$. We claim that therefore  
 \begin{align}
 | \Union_{\mu_0 Q \leq q \leq Q} \Union_{1 \leq a_2, \ldots, a_n \leq q} J^*(q;a_2,\ldots, a_n) |
 & \geq 3^{-(n-1)} c_4^{n-1}| \Union_{\mu_0 Q \leq q \leq Q} \Union_{1 \leq a_2, \ldots, a_n \leq q} J(q;a_2,\ldots, a_n) | \label{JJ}\\
 & \geq 3^{-(n-1)} c_4^{n-1}(1/2). \nonumber
 \end{align}
We assume for the moment that this is true, and verify it in Lemma \ref{lemma_box} below, as a consequence of the Vitali covering lemma.  

We now uniformize the lengths of the intervals. Let $I(q;a_2,\ldots, a_n)$ denote the product over $j=2,\ldots, n$ of the intervals centered at $2\pi a_j/q$ of length $4\pi c_4((\mu_0 Q) Q^{1/(n-1)})^{-1}$. 
Note that for each $\mu_0 Q \leq q \leq Q$, and $a_2,\ldots, a_n$, the box $I(q;a_2,\ldots, a_n)$ contains $J^*(q;a_2,\ldots, a_n)$. Thus 
\[ | \Union_{\mu_0 Q \leq q \leq Q} \Union_{1 \leq a_2,\ldots, a_n \leq q}  I(q;a_2,\ldots, a_n) | \geq c_4^{n-1}3^{-(n-1)}(1/2).\]

We now rescale the set defined by the union on the left-hand side in order to achieve the congruence conditions on $q, a_j$. Let $\mathcal{I}$ denote the union on the left-hand side. Every point $y' \in \mathcal{I}$ has a choice of $\mu_0 Q \leq q \leq Q$ and $1 \leq a_2,\ldots, a_n \leq q$ such that  $|y_j - 2\pi a_j/q| \leq 2\pi c_4 ((\mu_0 Q)Q^{1/(n-1)})^{-1}$ for $j=2,\ldots, n$. 
Thus by rescaling by a factor of 2, every $y' \in \mathcal{I}$ has a choice of $\mu_0 Q \leq q \leq Q$ and $1 \leq a_2,\ldots, a_n \leq q$ such that $|y_j/2 - 2\pi a_j/2q| \leq \pi c_4 ((\mu_0 Q)Q^{1/(n-1)})^{-1}$ for $j=2,\ldots, n$. Next we rewrite $2\pi a_j/2q = 2\pi (2a_j)/(4q)$, and set $q' = 4q$ and $a_j' = 2a_j$. Let $\mathcal{I}'$ denote the set $\mathcal{I}$ rescaled by $1/2$ in every coordinate, so that 
\[ |\mathcal{I}'| \geq 2^{-(n-1)} |\mathcal{I}| \geq c_4^{n-1}3^{-(n-1)}2^{-n}.\]
We can conclude that every $y' \in \mathcal{I}'$ has a choice of $4\mu_0 Q \leq q' \leq 4Q$ with $q' \con 0 \modd{4}$, and $2 \leq a_2',\ldots, a_n' \leq 2q$ with each $a_j' \con 0 \modd{2}$, such that $|y_j' - 2\pi a_j'/q'| \leq \pi c_4 ((\mu_0 Q)Q^{1/(n-1)})^{-1}$ for $j=2,\ldots, n$. 

All that remains to complete the proof of property (IIb) is a final lemma, which suffices to verify (\ref{JJ}).
\begin{lemma}\label{lemma_box}
Let $\{B_j\}_{j \in J}$ be a finite collection of cubes in $\R^m$. Fix a constant $0<c<1$ and for each $j$ let $B_j^*$ be the cube with the same center but with each side-length rescaled by $c$. Then 
\[| \Union_{j \in J} B_j^* | \geq c^{m} 3^{-m}|\Union_{j \in J} B_j|.\]
\end{lemma}
\begin{proof}
For each $j$, $|B_j^*| = c^m |B_j|$. If the union $\union_j B_j$ is disjoint, then $| \Union_{j \in J} B_j^* | = c^{m}  |\Union_{j \in J} B_j|$. Otherwise, by the Vitali covering lemma \cite[Ch. I \S 3.1 Lemma 1]{SteinHA}, there exists a disjoint subcollection  $\{B_{j_i}\}_{j_i \in J'}$  of $\{B_j\}_{j \in J}$ such that 
$ | \Union_{j_i \in J'} B_{j_i}| \geq c_m | \Union_{j \in J} B_j|,$
where we may take $c_m=3^{-m}$. 
Then the lemma holds, since we can apply the case for disjoint collections:
\[ | \Union_{j \in J} B_j^* | \geq | \Union_{j_i \in J'} B_{j_i}^* | = c^m | \Union_{j_i \in J'} B_{j_i}| \geq c^m 3^{-m} | \Union_{j \in J} B_j|.\]
\end{proof}

 \subsection{Property (III): volume of $\Omega^*$}
We now prove property (III) for the measure of $\Omega^*$. We first consider a 1-dimensional model problem, since we can later work coordinate-by-coordinate. 
Let $c_1>0$ be a small fixed constant and $M>0$ a large real scaling factor (sufficiently large that $Mc_1 > 2\pi$). Let $\iota$ denote the map $\iota : \R \maps \T \simeq [0,2\pi]$ that maps a real number to its image modulo $2\pi$. Given any   set  $S_0 \subset \T$ (in our case a union of intervals), we can define by periodicity  a set $S_1 \subset [-Mc_1,Mc_1]$ such that $\iota(S_1) = S_0$, and $S_1$ contains at least $2\lfloor Mc_1/2\pi \rfloor$ shifted copies of $S_0$. In particular, in measure $|S_1| \geq2 \lfloor Mc_1/2\pi \rfloor|S_0|$. Now let $r$ (for rescale) denote the map $r: \R \maps \R$ such that $r(x) = Mx$. Then given such a set $S_1$, we can define a set $S_2 \subset [-c_1,c_1]$ such that $r(S_2) = S_1$, and naturally the measure of $S_2$ is $|S_2| = |S_1|/M$. Composing these two processes, given a set $S_0$ in $[0,2\pi]$ we can construct a set $S_2$ in $\R$ with 
image $\iota \circ r (S_2) = S_0$, such that in measure $|S_2| \geq M^{-1} 2 \lfloor Mc_1/2\pi \rfloor|S_0| \geq (c_1/2\pi) |S_0|$, say.
Note that this lower bound is ultimately independent of $M$.

Similarly, this argument can be adapted to construct a set $S_1 \subset [-Mc_1,-Mc_1/2]$ containing at least $\lfloor Mc_1/(2 \cdot 2\pi)\rfloor$ copies of $S_0$, and then a set $S_2 \subset [-c_1,-c_1/2]$ of measure  $\geq (c_1/8\pi) |S_0|$ such that $\iota \circ r(S_2) = S_0$. 

We apply this coordinate-by-coordinate to $\Omega \subset [0,2\pi]^n$ (which is a union of products of intervals). We use the scaling factor $M=L^2/2R$ in the first coordinate, and $M=L$ in the $j$-th coordinate for $j=2,\ldots, n$.
Here we use the assumptions that $L = o(R)$ and $R = o(L^2)$ so that for all sufficiently large $R$ relative to an absolute constant, $M$ is sufficiently large relative to $c_1$ in each case.  Thus given $\Omega \subset [0,2\pi]^n \simeq \T^n$, we construct a set $\Omega^*\subseteq  [-c_1,-c_1/2] \times [-c_1,c_1]^{n-1}$ with $|\Omega^*| \geq c_1' |\Omega|$, where $c_1'$ is a positive constant depending only on $c_1,n$. 

We have constructed the sets $\Omega$ and $\Omega^*$, and verified properties (II) and (III). Next, we turn to verifying property (I).

 \section{Evaluating the arithmetic contribution}\label{sec_evaluate}
 Our goal in this section is to prove property (I), namely the identity (\ref{S_E3}) with the bound (\ref{E3_E4}) for the error term.
Since the sum $S(x',t;2R/L)$ factors into 1-dimensional sums, it suffices to work one coordinate at a time. 
Suppose that $x = (x_1,x') \in \Omega^*$ and correspondingly $y \in \Omega$, with corresponding $q$. 
Recalling the definition (\ref{Wt_dfn_tilde}), we may equivalently write, for any $ R/L \leq u \leq 2R/L$, the sum
\[ S_j(u)  = \sum_{R/L \leq m_j < u} e( m_j y_j + m_j^2(y_1 + s) ) .\]
Then to prove (I) it suffices to prove that for each $2 \leq j \leq n$, 
\beq\label{single_sum}
S_j(2R/L)= \sum_{R/L \leq m_j < 2R/L} e( m_j y_j + m_j^2(y_1 + s) ) 
	 =  \frac{\sqrt{2}R}{Lq^{1/2}}  + E_j(5)
	 \eeq
	 in which 
	 \beq\label{Ej5}
	  |E_j(5)| \leq C_5(c_4 + Q^{-\Del_0/2})  \frac{R}{LQ^{1/2}}  
	  \eeq
	 for some constant $C_5 = C_5(n,\Del_0,\mu_0)$.
Then to compute $|S(x',t;2R/L)|$ we multiply together (\ref{single_sum}) for $j=2,\ldots, n-1$ to get a main term of size 
$(\sqrt{2} R/Lq^{1/2})^{n-1}$ plus an error term that is of the form 
\[ 
 \sum_{\ell=0}^{n-2}C_\ell \left(  \frac{\sqrt{2} R}{Lq^{1/2}} \right)^{\ell} \left(   C_5 ( c_4 + Q^{-\Del_0/2})  \frac{R}{LQ^{1/2} }  \right)^{n-1-\ell}
\leq 
\left(  \frac{  R}{LQ^{1/2}} \right)^{n-1} \cdot \sum_{\ell=0}^{n-2}C_\ell'    C_5^{n-1-\ell} ( c_4 + Q^{-\Del_0/2})^{n-1-\ell}  ,
\]
for some combinatorial constants $C_\ell$ and  $C_\ell' = C_\ell' (C_\ell, \mu_0,n)$.
Under our assumptions, $(c_4 + Q^{-\Del_0/2})<1$ for all sufficiently large $R$, so that this term contributes the most when $\ell=n-2$.
Upon recalling from (\ref{RLQ_try3}) that $Q \geq (R/L)^{\ep_1}$ for some $\ep_1>0$, this is bounded above by the error term stated in (\ref{E3_E4}), with $C_3$ depending on $C_5,n, \mu_0$.

 \subsection{Proof of property (I)}\label{sec_main_term1}
We now prove (\ref{single_sum}). For future reference when bounding the error term $E(2)$  previously encountered in (\ref{E2_bound}), we furthermore prove a general result for $S_j(u)$ for any $R/L \leq u \leq 2R/L$.
Fix $x = (x_1,x') \in \Omega^*$, the corresponding $y \in \Omega$, the corresponding denominator $q$, and sums $S_j(u)$
for $j=2,\ldots, n$. 
In total, for any $R/L \leq u \leq 2R/L$, we will show
\beq\label{sup_Sj}
  |S_j(u)|  = \frac{\sqrt{2}(u-R/L)}{q^{1/2}}  + E_j(u;3) + E_j(4),
\eeq
where the error terms satisfy the bounds (\ref{Ej3}) and (\ref{enlarge_E4}) below, respectively. 
In  the case $u=2R/L$, (\ref{sup_Sj}) proves (\ref{single_sum}).

 Recall from
(\ref{sum_j_2}) that after approximating $y_j$ by $2\pi a_j/q$ in the sum,
\[ S_j(u)= \tilde{S}_j(u) + E_j(u;3).\]
 So in particular we evaluate $\tilde{S}_j(u)$. Recall $x \in \Omega^*$, with corresponding values of $q \con 0 \modd{4}$ and $(a_1,q)=1$, $a_j \con 0 \modd{2}$ for $j=2,\ldots, n$. 
 For a fixed $R/L \leq u  \leq 2R/L$, with $s$   chosen above in (\ref{ys_choice}), the sum
 \[ \tilde{S}_j(u) =   \sum_{R/L \leq m_j < u} e( m_j (2\pi a_j/q)+ m_j^2(y_1 + s) )  \]
   is equal to 
 $ \left\lfloor (u-R/L)/q \right\rfloor G(a_1,a_j;q)$, plus 
 possibly an incomplete sum of length $< q$.   Lemma \ref{lemma_Gauss} case (3) shows that $|G(a_1,a_j;q)|=\sqrt{2}q^{1/2}$, while the incomplete sum is dominated by
 \beq\label{incomplete_sum}
\sup_{\bstack{1 \leq u \leq u' \leq q}{u'-u<q}}  \left|\sum_{u \leq m_j  \leq  u'}e(2\pi \frac{a_j}{q} m_j + 2\pi \frac{a_1}{q} m_j^2) \right| \leq 2C_0 q^{1/2}  (\log q)^{1/2},
   \eeq
as a consequence of the Weyl bound (Lemma \ref{lemma_Weyl}) with $N=u'-u<q$.
This proves that 
\beq\label{tilde_S_j_id}
 |\tilde{S}_j(u) | =   \left\lfloor \frac{u-R/L}{q}\right\rfloor \sqrt{2}q^{1/2}  + E_j'(4)  = \frac{\sqrt{2}(u-R/L)}{q^{1/2}} + E_j(4), 
 \eeq
 say,
with $|E_j'(4)| \leq 2 C_0 q^{1/2} (\log q)^{1/2}$. In the second identity, we have written $ \left\lfloor (u-R/L)/q \right\rfloor \sqrt{2}q^{1/2}  = \sqrt{2}(u-R/L)/q^{1/2} + O(q^{1/2})$, and we obtain the error term bound
$
 |E_j(4)| \leq (2C_0 +2) q^{1/2} (\log q)^{1/2} .
$
In particular, recall from (\ref{RLQ_try2}) that $R/L \geq Q^{1+ \Del_0}$.
For all  $4\mu_0 Q \leq q \leq 4Q$, 
\[
q^{1/2} (\log q)^{1/2} \leq C_{\Del_0}' q^{1/2  + \Del_0/2}  \leq C_{\Del_0,\mu_0} \frac{Q^{1+ \Del_0}}{Q^{1/2}} Q^{-\Del_0/2}
	\leq C_{\Del_0,\mu_0}' \frac{R}{LQ^{1/2}} Q^{-\Del_0/2}.
\]
Consequently
\beq\label{enlarge_E4}
 |E_j(4)| \leq  C_{\Del_0,\mu_0}'' \frac{R}{LQ^{1/2}} Q^{-\Del_0/2}   .
 \eeq
 
 We now also use this to bound $E_j(u;3)$.
In particular,   we can derive the crude upper bound
\[ \sup_{R/L \leq u \leq 2R/L} |\tilde{S}_j(u) | \leq \frac{\sqrt{2}R}{Lq^{1/2}} + |E_j(4)| \leq C_{\Del_0,\mu_0}''' \frac{R}{LQ^{1/2}} .
\]
We insert this in (\ref{sum_est_1_0}) and recall that $V = \pi ((\mu_0Q)Q^{1/(n-1)})^{-1} \leq C_3'(R/L)^{-1}$ to conclude that 
for any $R/L \leq u \leq 2R/L$,
\beq\label{Ej3}
|E_j(u;3)| \leq c_4 C_{\Del_0,\mu_0}''' C_3'\frac{R}{LQ^{1/2}} 
\eeq
for some constant $C_{\Del_0,\mu_0}''' = C_{\Del_0,\mu_0}'''(n)$.
This verifies (\ref{sup_Sj}) and hence (\ref{single_sum}), and we have completed the proof of property (I).

\subsection{Completing the bound for $E(2)$}
Second, in order to bound $S_j(u)$ as it appears in $E(2)$ in (\ref{E2_bound}), we rewrite (\ref{sup_Sj}) as the cruder upper bound
\beq\label{crude_Sju}
|S_j(u)| \leq C_{\Del_0,\mu_0}'''' \frac{R}{LQ^{1/2}},
\eeq
valid for all $R/L \leq u \leq 2R/L$, with some constant $C_{\Del_0,\mu_0}''''$.  
 We now apply this in (\ref{E2_bound}) to see that (\ref{E2_bound_wish}) holds, as desired, with a constant $C_2 = C_2(n,\Del_0,\mu_0)$.

\section{Final estimates and choice of parameters}\label{sec_parameters}
Our starting point for this section is the key result  of property (I) for $|S(x',t;2R/L)|$ in equation (\ref{S_E3}). We combine this with   the key result for  $|(e^{it\Del}f)(x)|$ in equation (\ref{conseq1}) of \S \ref{sec_reduce} to see that 
for every point $x \in \Omega^*$, there exists a choice of $t \in (0,1)$ and some $4\mu_0 Q \leq q \leq 4Q$ such that
\beq\label{conseq2}
  |(e^{it \Del} f)(x) |   \geq (1-c_0)^n  \left( \frac{\sqrt{2}R}{Lq^{1/2}}\right)^{n-1}  - (|E(1)| + |E(2)|  + |E(3)|).
  \eeq
  In particular, for each $4\mu_0 Q \leq q \leq 4Q$,
  \[ 
  (1-c_0)^n  \left( \frac{\sqrt{2}R}{Lq^{1/2}}\right)^{n-1}  \geq (1-c_0)^n 2^{-(n-1)/2} \left( \frac{ R}{LQ^{1/2}}\right)^{n-1}.
  \]
  
In this section, we will confirm that for each choice of $c_0 \leq c_0^*$ and $\del \leq \del_0^*$, with thresholds $c_0^*, \del_0^*$ specified in (\ref{c_del'}) below, 
there exists an absolute constant $R_0$ depending only on $n, \phi$ (and other  constants that we have chosen in terms of $n,\phi$), such that for all $R  \geq R_0$, we have 
\beq\label{conseq3_0}
 |E(1)| +|E(2)| + |E(3) |\leq (1/2) (1-c_0)^n 2^{-(n-1)/2}   \left( \frac{R}{L Q^{1/2}}\right)^{n-1} ,
 \eeq
so that 
\beq\label{conseq3}
  |(e^{it \Del} f)(x) |   \geq (1/2)(1-c_0)^n 2^{-(n-1)/2}  \left( \frac{R}{L Q^{1/2}}\right)^{n-1}.
  \eeq
After we bound the error terms $E(1),E(2),E(3)$, we will choose the parameters $S_1, L, Q$ appropriately, according to the constraints we have imposed so far.

\subsection{Bounding $E(1),E(2),E(3)$}
It is simple to bound $E(1)$ and $E(2)$, now that we have constructed the set $x\in \Omega^*$ and chosen $t$ for each $x$ accordingly. 
The key is to bound $W(t)$ and $S_j(u)$, as defined in (\ref{Wt_dfn}) and (\ref{Wt_dfn_tilde}), respectively. 
Fix $x \in \Omega^*$, with corresponding $y \in \Omega$ and associated denominator $q$. Recalling that we choose $t$ so that $y_1 + s = 2\pi a_1/q$,   for any $ R/L \leq u < 2R/L$ and  uniformly in $v \in [0,2\pi]$,
\[   \left|\sum_{R/L \leq m \leq u}  e(vm + m^2 (y_1 +s))   \right|  \leq  C_0\left( \frac{R}{Lq^{1/2}} + q^{1/2}\right) (\log q)^{1/2}
	\leq 4 C_{0,\mu_0} \frac{R}{LQ^{1/2}}(\log Q)^{1/2} ,
\]
upon recalling $R/L \geq Q \geq q/4$.
This bound suffices to treat $W(t)$.
We recall the upper bound for $|E(1)|$ in terms of $W(t)$ as stated in (\ref{E1_bound}), which now implies that
\[ |E(1)|   \leq   |t|  C_1 \| \hat{\phi} \|_{L^1}^{n-1}\left(4C_{0,\mu_0}   \frac{R}{LQ^{1/2}}(\log Q)^{1/2} \right)^{n-1}.
      \]

Next we recall the upper bound for $|E(2)|$ stated in (\ref{E2_bound}) in terms of  $S_j(u)$; also recall that we have already  verified that (\ref{E2_bound_wish}) holds.
In conclusion,   we have the upper bound
 \[ |E(1)| + |E(2)| \leq   \left( \frac{R}{LQ^{1/2}}\right)^{n-1} ( C_4R|t| + C_4|t|(\log Q)^{\frac{n-1}{2}})  \]
 where we take $C_4 =C_4(n, \Del_0,\mu_0,\phi)$ depending on our previous constants chosen in terms of these parameters. 
 
 Now we recall from (\ref{t_fact1}) that $|t| \leq \del_0/8R$, so that if $\del_0^*$ is chosen sufficiently small that $\del_0 \leq \del_0^* \leq  c_02^{-(n-1)/2}/C_4$, 
 we certainly have $C_4 R|t| \leq c_0 2^{-(n-1)/2}/8$.
 We also have $Q \leq R/L \leq R$ (as a crude upper bound), so  there exists a constant $R_3$ chosen appropriately large relative to $n,C_4$ such that 
  for all $R \geq R_3$, $C_4   |t| (\log Q)^{\frac{n-1}{2}} \leq C_4   (\del_0/8)R^{-1} (\log R)^{\frac{n-1}{2}}  \leq c_0 2^{-(n-1)/2}/8$. 
In total, under these conditions, we then have
 \beq\label{E1E2}
  |E(1)| + |E(2)| \leq      \left( \frac{R}{LQ^{1/2}} \right)^{n-1} 2^{-(n-1)/2}(c_0/4) .  
  \eeq
  
 We also recall the bound (\ref{E3_E4}) for $|E(3)|$, which holds for some $C_3 = C_3(n, \Del_0,\mu_0)$ and some $\ep_1>0$ as in (\ref{RLQ_try3}), and for any $c_4<1/2$ of our choice. 
 We now specify that we take 
 \beq\label{c4_choice}
 c_4 \leq \frac{2^{-(n-1)/2}}{8C_3} c_0.
 \eeq
There exists a constant $R_4$ chosen appropriately large relative to $\Del_0, C_3, \ep_1$, such that for all $R \geq R_4$, $C_3 (R/L)^{-\Del_0 \ep_1/2} \leq c_02^{-(n-1)/2}/8$.  Then under these conditions,
 \[ |E(3)| \leq    \left( \frac{R}{LQ^{1/2}} \right)^{n-1} 2^{-(n-1)/2}(c_0/4)  .
 \]
 Now finally we take  $R_0 = \max \{R_1,R_2,R_3, R_4\}$. We specify that 
   \beq\label{c_del'}
c_0^* \leq 2^{-n}, \qquad \del_0^*  \leq \min \{ \del_0(c_0),  c_02^{-(n-1)/2}/C_4\}.
\eeq
These restrictions depend only on $n,\phi$.
The former condition assures that for all $c_0 \leq c_0^*$ we have $c_0 \leq (1-c_0)^n$.
  Then for $R \geq R_0$ and under the conditions (\ref{c_del'}), we have shown that 
  \[  |E(1)| + |E(2)| + |E(3)| \leq  (1/2) 2^{-(n-1)/2} c_0  \left( \frac{R}{LQ^{1/2}} \right)^{n-1}  \leq (1/2)2^{-(n-1)/2} (1-c_0)^n  \left( \frac{R}{LQ^{1/2}} \right)^{n-1},
  \]
as claimed in (\ref{conseq3_0}).

\subsection{Heuristics to motivate choices for the parameters}
Recall our key goal inequality (\ref{Bou''}), which would follow from (\ref{goal_tilde_f}) under the assumption that $\Omega^*$ has positive measure independent of $R$. Our construction only shows that $|\Omega^*|$ is at least proportionate to $c_{\ep_0} Q^{-\ep_0},$ for any $\ep_0$ of our choice. In this setting, we will prove (\ref{Bou''}) directly.
Recall the computation of the norm $\|f\|_{L^2} = S_1^{-1/2} (R/L)^{\frac{n-1}{2}}  \| \phi \|^n_{L^2}$ from (\ref{f_norm_computation}),
as well as the lower bound (\ref{conseq3}) and the measure of $|\Omega^*|$. 
We can verify  (\ref{Bou''}) for each $s <  s^* :=n/(2(n+1))$ if we can show that for each such $s$, there is an $\ep_0$ small enough that
 \beq\label{end_game2}
 \left( \frac{R}{L Q^{1/2}}\right)^{n-1}  S_1^{1/2} (R/L)^{-(\frac{n-1}{2})} Q^{-\ep_0}  \geq   A_s R^{s'}
 \eeq
for some $s' > s$.
 (Here we may take  $A_s $ depending on $s, n, \phi, \Del_0, \mu_0, \ep_0$ and all previous constants we have chosen in terms of these parameters.)  
Our goal now is to choose $S_1,L,Q$ so as to verify (\ref{end_game2}), and also fit all the constraints we have previously imposed. 
 Then we will be able to conclude that for all $s <n/(2(n+1))$, for every $R \geq R_0$ we have constructed a Schwartz function $f=f_R$ so that $\hat{f}$ is supported in the annulus $A_n(R,4 \sqrt{n})$ and (\ref{Bou''}) holds, concluding the proof of the main theorem.

\emph{A priori} we aim to choose $L,S_1,Q$ in terms of $R$ so as to prove (\ref{end_game2}) for the largest value of $s$ possible; we will see that the limit of this construction  is $s< s^*=n/(2(n+1))$. We temporarily pretend that $\ep_0$ is zero, so that we can simplify our computations. Then once we have motivated our choices for $L,S_1,Q$, we will compute precisely.
Assign the notations $\lam, \sig,\kappa$ according to
\[ L=R^\lam, \qquad S_1 = R^\sig, \qquad Q = R^\kappa.
\]
 We summarize the key constraints on $\lam,\sig,\kappa$ as follows.

The truth of the core inequality (\ref{end_game2}) for some $s'>s$ is equivalent (assuming $\ep_0=0$ for the moment) to 
\beq\label{con_core}
s <  (n-1)/2 + \sig/2 - (\kappa + \lam)(n-1)/2  ,
 \eeq
 which we want to hold for $s$ as large as possible.
  We also have the constraint $\sig \leq 1/2$ from (\ref{sig}), and $1/2< \lam < 1$ from \S \ref{sec_3_ppties}.
From the conditions (\ref{UV_assps}) we have
\beq\label{lamkapsum} 2\lam + \kappa \geq 1 + \sig, \qquad \lam + \kappa( \frac{n}{n-1}) \geq 1.
\eeq
Using the linear combination of $1/(n-1)$ times the first inequality plus $1$ times the second inequality, we see that 
\beq\label{lam_kap}
 \lam + \kappa \geq \frac{n+ \sig}{n+1}.
 \eeq
The upper bound in (\ref{con_core}) will be largest when $\lam + \kappa$ is smallest, so it is optimal to choose $\lam,\kappa$ so that equality holds in (\ref{lam_kap}), in which case (\ref{con_core}) will hold for all
\beq\label{s_limit}
 s< \frac{ n-1 + 2\sig}{2(n+1)}.
 \eeq
Since we want to take $s$ as large as possible, this motivates us to take $\sig$ as  large as possible, that is 
\[ \sig=1/2.\] 
This illuminates why the largest exponent we could win from Bourgain's construction is $s < s^* = n/(2(n+1))$, no matter how we choose $L,Q$.

Finally we need to choose $\kappa,\lam$ so that the two constraints in (\ref{lamkapsum}) hold, and equality holds in (\ref{lam_kap}). 
 The first two constraints represent a region in the first quadrant of the $(\kappa,\lam)$-plane bounded by two lines, and these two lines intersect the line representing equality in (\ref{lam_kap}) in precisely one point, namely $(\kappa,\lam) = (\frac{n-1}{2(n+1)},\frac{n+2}{2(n+1)})$. Thus this is the unique choice of $\lam,\kappa$ that meets all our requirements. 
These choices correspond to defining
\beq\label{LR_guess}
S_1 = R^{1/2}, \qquad L = R^{\frac{n+2}{2(n+1)}}, \qquad Q = R^{\frac{n-1}{2(n+1)}}.
\eeq
This corresponds to the value $\Del_0 = 1/(n-1)$ and $\ep_1 = 1 /(1+\Del_0)$ in  conditions (\ref{RLQ_try2}) and (\ref{RLQ_try3}). These are also the choices that Bourgain states.

 \subsection{Precise   conclusions}
 Having motivated our choices for $S_1,L,Q$, we   perform the final verifications precisely. 
We fix any $s < s^* =n/(2(n+1))$, and we aim to show that (\ref{end_game2}) holds for some $s'> s$, where we may take $\ep_0$ as small as we like. This will hold if $\sig, \kappa,\lam,\ep_0$ are such that  
\beq\label{con_core'}
s <  (n-1)/2 + \sig/2 - (\kappa + \lam)(n-1)/2 - \ep_0 \kappa.
 \eeq
 The relation (\ref{lam_kap}) still holds, and we will choose $\lam,\kappa$ so that equality holds in (\ref{lam_kap}),
 so that (\ref{con_core'}) becomes the relation
\[
  s< \frac{ n-1 + 2\sig - 2\ep_0 \kappa (n+1)}{2(n+1)}.
\]
To make the right-hand side as large as possible we choose $\sig=1/2$, obtaining
  \beq\label{s_limit''}
  s< \frac{ n  - 2\ep_0 \kappa (n+1)}{2(n+1)}.
 \eeq
 We choose $\lam,\kappa$ as before (depending only on $n$), and then take $\ep_0$ arbitrarily small. 
 
This implies that for every $s< n/(2(n+1))$, the following holds. 
There exists a constant $C$ depending only on $n$ and a constant $R_0$ depending only on $n,\phi$ such that for every integer $R \geq R_0$ we can construct a Schwartz function $f=f_R$ with $\hat{f}$ supported in the annulus $A_n(R,4\sqrt{n})$ such that (\ref{end_game2}) and hence (\ref{Bou''}) holds.
 This completes the proof of Theorem 
 \ref{thm_restate} and hence of Theorem \ref{thm_Bou}.
 
 \section*{Appendix A: convergence results}\label{sec_conv}
 \setcounter{equation}{0}
 \renewcommand{\theequation}{A.\arabic{equation}}
 For the benefit of a general audience, we recall the relationship between  maximal functions and pointwise convergence;  
these ideas underly many results in the literature, and similar expositions   can be found for example in \cite[Thm. 5]{Sjo87} or \cite[Appendix C]{BBCR11}.

 \subsection*{Positive results}

 If $f$ is a Schwartz function, then
\[( e^{it \Del} f)(x) = (T_t f)(x)= \frac{1}{(2\pi)^n} \int_{\R^n} \hat{f}(\xi) e^{ i (\xi \cdot x + |\xi|^2 t)} d\xi.\]
We may also write $T_t f = K_t * f$ with $K_t(x) = t^{-n/2} K(x/t^{1/2})$ where $K$ is the Fourier transform of $e^{i|\xi|^2}$, that is
$K(x) = c_n e^{-i |x|^2/4}$ for a constant $c_n$.
For $f \in H^s(\R^n)$, for any fixed $t>0$ we can define $T_t f$ as the limit in the $L^2$ sense, since $T_t$ is a bounded operator on $L^2(\R^n)$, or alternatively we can fix a Schwartz function $\psi(x)$ and define 
\[ (T_t f)(x)= \lim_{N \maps \infty} \frac{1}{(2\pi)^n} \int_{\R^n}\psi(|\xi|/N) \hat{f}(\xi) e^{ i (\xi \cdot x + |\xi|^2 t)} d\xi,\]
which agrees with the $L^2$ limit pointwise a.e. (see e.g. \cite{BBCR11}).

If $f$ is Schwartz, upon applying the integral representation to  the difference $T_tf(x ) -f(x)$, we see that 
\beq\label{ptwise_app}
 \lim_{t \maps 0} (e^{it \Del} f)(x) =f(x) 
 \eeq
  for every $x$. 
Positive results on the pointwise a.e. convergence for functions $ f \in H^s(\R^n)$  (for an appropriate fixed $s$) proceed by bounding the maximal operator defined by 
\[ T^* f(x) = \sup_{0<t<1}|T_tf(x)|.\]
For example, in order to prove that  pointwise convergence (\ref{ptwise_app}) holds for almost every $x$, for all functions $f \in H^s(\R^n)$, 
 it suffices to prove that for all Schwartz functions $f$,
\[  \| T^* f\|_{L^2(B_n(0,1))} \leq A \|f\|_{H^s(\R^n)}.\]
 
We see this as follows, recalling  the method of \cite[Thm. 5]{Sjo87}. Suppose that $f \in H^s(\R^n)$. 
To show that (\ref{ptwise_app}) holds for a.e. $x$ it would suffice to show that for every ball $B$ of finite radius, 
\beq\label{B_int}
 \int_{B} \limsup_{t \maps 0} |T_t f(x)  - f(x)|^2 dx =0.
 \eeq
For any $\ep>0$ there exists a Schwartz function $g$  with $\| f - g\|_{H^s(\R^n)} \leq \ep$, so that
\[ \limsup_{t \maps 0} |T_t f(x)  - f(x)| \leq \limsup_{t \maps 0} | (T_t f - T_t g)(x)| + | f(x) - g(x)|
	\leq T^* (f-g)(x) + |f(x) - g(x)|.\]
Also note that $\|f - g\|_{L^2(B)}  \leq  \|f - g\|_{L^2(\R^n)}  = (2\pi)^{-n/2}  \| (f - g)\hat{\:}\|_{L^2(\R^n)} \leq \| f-g \|_{H^s(\R^n)}  \leq \ep$.
Thus we have
\beq\label{B_int'}
  (\int_{B} \limsup_{t \maps 0} |T_t f(x)  - f(x)|^2 dx)^{1/2} \leq \|T^* (f-g)\|_{L^2(B)} + \ep.
  \eeq
Thus it suffices to show   that  for every ball $B$ of finite radius, there exists a constant $C_B$ such that for all $h \in H^s(\R^n)$,
\beq\label{allB}
\| T^* h \|_{L^2(B)}  \leq C_B \|h\|_{H^s(\R^n)} .
\eeq
We would then apply this with $h=f-g$ to conclude that (\ref{B_int'}) is at most $(C_B+1)\ep$, which suffices. 

We now show that we can conclude (\ref{allB}) holds if we can show it for all Schwartz functions.  
Fix $f \in H^s(\R^n)$ and  a sequence of Schwartz functions $f_m$ such that $\|f_m - f\|_{H^s(\R^n)} \maps 0$.
By Fatou's lemma (after passing to a subsequence, if necessary), and our assumed inequality for Schwartz functions,
\beq\label{B_comp}
 \int_B |T^* f(x)|^2 dx \leq \liminf_{m \maps \infty} \int_B |T^* f_m(x)|^2 dx 
	\leq \liminf_{m \maps \infty} C_B^2 \|f_m\|_{H^s}^2  = C_B^2 \|f\|_{H^s}^2,\eeq
as desired.
  Indeed it even suffices to prove (\ref{allB}) only for the ball $B=B_n(0,1)$; certainly 
we can deduce (\ref{allB}) for any finite radius ball $B$ if we can show it for all unit balls $B$. To see that it suffices in particular to consider the unit ball $B_n(0,1)$ at the origin, we only note that for any fixed shift $u \in \R^n$, $(T_t f)(x+u) = (T_t g_u) (x)$ where $g_u(x)$ is defined by $\hat{g_u} (\xi)= \hat{f}(\xi) e^{i \xi \cdot u}$, so that $\|g_u\|_{H^s(\R^n)} = \|f\|_{H^s(\R^n)}$.

\subsection*{Negative results}
Counterexamples (such as e.g. \cite{Bou13,Bou16,BBCR11,LucRog17,LucRog19a,LucRog19,DemGuo16}) to pointwise a.e. convergence in (\ref{ptwise_app}) for functions $f \in H^s(\R^n)$ (or other spaces)  rely on a type of converse to the above argument, which holds due to a maximal principle of Stein \cite{Ste61}.
We lay out the necessary steps here, using the version of the maximal principle stated in \cite[Ch. X \S 3.4]{SteinHA}. 
 In what follows we will let $C_n$ denote a constant that depends on $n$, which may change from one instance to the next, and similarly for $C_s, C_{n,s}$ and so forth.  We thank Jongchon Kim for suggesting the presentation we follow here.

We claim: if for a given $s>0$ it is true that for all $f \in H^s(\R^n)$, (\ref{ptwise_app}) holds for a.e. $x$, then it is true that for all $f \in H^s(\R^n)$,
\beq\label{L2}
\| T^* f \|_{L^2(B_n(0,1))} \leq C_{s}  \|f\|_{H^s(\R^n)}.
\eeq
Once  (\ref{L2}) holds, then since $\| T^* f \|_{L^1(B_n(0,1))}\leq C_n \| T^* f \|_{L^2(B_n(0,1))}$,  all such $f$ must also satisfy
\beq\label{L1}
\| T^* f \|_{L^1(B_n(0,1))} \leq C_{n,s} \|f\|_{H^s(\R^n)}.
\eeq
Consequently, if we can show for a given $s$ that (\ref{L1}) is violated by some function, the pointwise convergence result (\ref{ptwise_app}) must also fail.
This principle underlies the counterexample of Bourgain that is the subject of the present note.
In particular, if $\supp \hat{f} \subseteq A_n(R,C_n)$, the right-hand side of (\ref{L1}) is comparable to $C'_{n,s} R^{s} \|f\|_{L^2(\R^n)}$,  so that if we prove (\ref{Bou''})   for some $s'>s$, this provides a violation of (\ref{L1}).
 
To prove the claim  (\ref{L2}), we will  use  the following notation. Recall the operator $G_s$ defined by
$(G_s g) \hat{\;}(\xi) = (1+|\xi|^2)^{-s/2}\hat{g}(\xi)$, so that $g \in L^2$ if and only if $G_s g \in H^s$.  Then (\ref{ptwise_app}) is equivalent to the hypothesis that   for all $g \in L^2(\R^n)$,
\beq\label{ptwise_app'}
\lim_{t \maps 0} T_t G_s g(x) = G_s g(x) \quad \text{ holds for a.e. $x$}.
\eeq
There are two steps:\\
(1)  Maximal Principle:  if for a given $s>0$ it is true that for all $g \in L^2(\R^n)$, (\ref{ptwise_app'}) holds for a.e. $x$,  then a weak-type $L^2$ bound holds, namely that for all $g \in L^2(\R^n)$, for all $\al>0$
\beq\label{wL2}
|\{ x \in B_n(0,1) : (T^* G_s g)(x) > \al \}| \leq \frac{A_s}{\al^2} \|g\|_{L^2(\R^n)}^2.\eeq
(2) A H\"older inequality for Lorentz spaces on the unit ball:
\[ \| T^* G_s g\|_{L^1(B_n(0,1))} \leq C_n \| T^*G_s  g\|_{L^{2,\infty}(B_n(0,1))}.\]
Applying (\ref{wL2}) shows that the right-hand side of point (2) is bounded above by $C_n A_s^{1/2} \| g\|_{L^2(\R^n)}.$
Finally, given $f \in H^s(\R^n)$, let $g = G_{-s} f$ and apply the above inequalities to $g$, to conclude that (\ref{L2}) holds for $f$.

We prove point (1) by the maximal principle.
We first note that instead of considering the family $\{ T_t \}_{t >0}$ with continuous parameter $t>0$, we may apply the maximal principle to the family $\{ T_{t_k} \}_{k >0}$ with $k$ varying over the discrete set $\mathbb{N}$, and with $\{t_k\}_k$ an enumeration of the positive rationals, and with corresponding maximal function $T^{(*)} := \sup_k T_{t_k}$. 
This is because as observed in \cite[\S 12]{Ste61}, $\sup_{t>0} |T_t f(x)| = \sup_{k} |T_{t_k} f(x)|$ since for each $x$, $(T_t f)(x)$ is continuous in $t$. Thus we may state our conclusions below for the maximal operator $T^{(*)}$, and they then also hold for $T^{*}$.
 Moreover, under the assumption that $\lim_{t \maps 0} (T_tf)(x)$ exists, we have $\limsup_{k \maps \infty} |T_{t_k}f(x)| < \infty$, which will be used below.

Fix $s>0$. We check that $\{ T_{t_k} G_s \}_k$ satisfies the criteria of the maximal principle.
 For each $k$ the operator $T_{t_k}G_s $ is bounded on $L^2(\R^n)$, and so in particular satisfies the property that if $g_m \maps g$ in $L^2(\R^n)$ then for each fixed $k$, $T_{t_k}G_s  (g_m) \maps T_{t_k} G_s (g)$ in measure. 
 Since we are assuming that for all $g \in L^2(\R^n)$ we have $\lim_{k \maps \infty} T_{t_k} G_s g(x) = g(x)$ for pointwise a.e. $x$, then certainly $T^{(*)}G_s g(x) < \infty$ on a set of positive measure. Thus by \cite[Ch. X \S 3.4]{SteinHA}, for the compact set $B_n(0,1)$ there exists a constant $A_s$ such that for all $g \in L^2(\R^n)$, for all $\al>0$,
\[
 |\{ x \in B_n(0,1) : (T^{(*)} G_s g)(x) > \al \}| \leq \frac{A_s}{\al^2} \|g\|_{L^2(\R^n)}^2.
\]
 This concludes the proof of point (1).

We prove point (2) by a direct argument. Indeed, for a finite measure set $B$, any function $f \in L^{2,\infty}(B)$ satisfies
$\|f\|_{L^1(B)} \leq 2 |B|^{1/2} \|f\|_{L^{2,\infty}(B)}.$
To see this, let $\lam(t) = |\{ x \in B: |f(x)|>t\}|$ so that
\[ \|f\|_{L^1(B)} = \int_0^\infty \lam(t)dt = \int_0^A \lam (t) dt + \int_A^\infty t^2 \lam (t) \frac{dt}{t^2},\]
for any $A>0$ of our choice.
The first term is bounded by $A|B|$, while the second term is bounded by $A^{-1}\|f\|_{L^{2,\infty}(B)}^2 $, so the inequality follows from choosing $A= B^{-1/2} \|f\|_{L^{2,\infty}(B)}$. This suffices to prove point (2).

 \section*{Appendix B: Classical number theoretic facts}\label{sec_exp}

\setcounter{equation}{0}
\renewcommand{\theequation}{B.\arabic{equation}}
 The proofs of Lemmas \ref{lemma_Gauss} and \ref{lemma_Weyl} both follow from what is commonly called the method of $S\overline{S}$ in number theory and $TT^*$ in harmonic analysis. Many standard texts, such as \cite{IK}, contain similar proofs.
 
\subsection*{Proof of Lemma \ref{lemma_Gauss}}
Note that in the sum $G(a,b;q)$, and in the complete sums that follow, we can sum over any complete set of residues modulo $q$, and in particular over any $q$ consecutive integers. We compute that 
\[ |G(a,b;q)|^2 = G(a,b;q) \overline{G(a,b;q)} = \sum_m \sum_n e^{2\pi i (m^2 a/q + mb/q)} e^{-2\pi i (n^2 a/q +nb/q)}.\]
We replace $m$ by $n+\ell$, so that
\beq\label{double_sum0}
|G(a,b;q)|^2 = \sum_{1 \leq \ell \leq q} \left( \sum_{1 \leq n \leq q} e^{2\pi i ( (n+\ell)^2a/q - n^2a/q + \ell b/q)} \right)
	=\sum_{\ell=1}^q e^{2\pi i (\ell^2 a/q+\ell b/q)} \left( \sum_{n=1}^q e^{2\pi i n(2\ell a/q)} \right).
\eeq
In the right-most inner sum over $n$, the phase is linear in $n$, and this is the main point of the proof. Indeed, 
for any real number $\theta$ such that $e^{2\pi i q \theta}=1$ but $e^{2\pi i \theta }\neq 1$ (in which case $e^{2 \pi i \theta}$ is a nontrivial $q$-th root of unity),
$
\sum_{n=1}^q e^{2\pi i n \theta} =0 .
$
 On the other hand, if $e^{2\pi i \theta} =1$ (which is true precisely when $\theta$ is an integer),
$ \sum_{n=1}^qe^{2\pi i n \theta} =q.$

 We apply this to (\ref{double_sum0}) with $\theta = 2\ell a/q$, recalling that $(a,q)=1$. We consider first the case when $q$ is odd. Then the right-most inner sum over $n$ in (\ref{double_sum0}) vanishes except when $\ell =q$, and then its value is $q$. Thus $|G|^2 = q$ in this case. 
If on the other hand $q$ is even,
 the right-most inner sum over $n$ in (\ref{double_sum0}) is non-zero for precisely two values of $\ell$:  $\ell=q$ and $\ell=q/2$. Summing the two resulting values,  
 \[ |G(a,b;q)|^2 = q ( e^{2\pi i (aq +b)} + e^{2\pi i (aq/4 + b/2)}) = q(1 + e^{\pi i (q/2+b))});\]
 in the last identity we used the fact that $(a,q)=1$ so that $a\con 1 \modd{2}$.
 Thus $|G(a,b; q)|^2$ is determined by the parity of $q/2+b$: it vanishes if $q/2+b \con 1 \modd{2}$; it equals $2q$ if $q/2+b \con 0 \modd{2}$. 
 When $q \con 2 \modd{4}$, the first case occurs when $b$ is even and the second case when $b$ is odd.
 When $q \con 0 \modd{4}$, the first case occurs when $b$ is odd and the second case when $b$ is even.

 \subsection*{Proof of Lemma \ref{lemma_Weyl}}
To prove Lemma \ref{lemma_Weyl}, we will proceed via an argument similar to that used in Lemma \ref{lemma_Gauss}, squaring and differencing in order to reduce to the case of a linear exponential sum. This we estimate via the standard result that for any real number $\theta$, for any $N \geq 1$, 
\beq\label{linear_Weyl}
 \sum_{M \leq n < M+ N} e^{2\pi i \theta n}   \leq  \min\{ N, (2\| \theta\|)^{-1}\},
 \eeq
in which $\|\theta\|$ denotes the distance from $\theta$ to the nearest integer. The trivial bound $N$ applies when $\theta =0$; otherwise, 
the sum is equal to $|\sin (\pi \theta N) /\sin (\pi \theta)|$ in absolute value, from which the estimate follows, using $|\sin (\pi \theta)| \geq 2\|\theta\|$.

Now we turn to the quadratic sum in question in Lemma \ref{lemma_Weyl}, which we denote by $S$.   First we suppose that $N \leq q/4$. We compute $|S|^2$ and re-write $n$ as $m+h$, so that
\begin{eqnarray*}
 |S|^2 &=& \sum_{M \leq n,m <M+ N} e^{2\pi i (\al (n^2 -m^2)+ \be (n-m))} 
  = \sum_{|h|< N} e^{2\pi i (\al h^2 + \be h)} \sum_{M \leq m, m+h < M+N} e^{2\pi i (2\al mh)}.
	 \end{eqnarray*}
Applying the trivial bound to the $h=0$ term and (\ref{linear_Weyl}) to the last sum, 
\[ |S|^2 \leq N + 2 \sum_{1 \leq h < N} \min\{N,(2\|\al 2h\|)^{-1}\} \leq N + 2\sum_{1 \leq h < 2N} \min\{N,  (2\|\al h\|)^{-1}\}.\]
We recall that $|\al - a/q| \leq 1/q^2$, with $(a,q)=1$. For any $1 \leq h <2N \leq q/2$, we claim that 
$\| \al h\| \geq \frac{1}{2} \| a h/q\|$. This is because since $q \ndiv h$ and $(a,q)=1$, we know that $ah/q$ is at least $1/q$ from the nearest integer. In combination with the fact that $|\al h - ah/q| \leq 2N/q^2 \leq 1/2q$  so that $\al h$ is at least $1/2q$ from the nearest integer, this suffices.

 Thus
\[ |S|^2 \leq   N + 2\sum_{1 \leq h < 2N \leq q/2} \min(N, \|a h/q\|^{-1}) \leq N +  2 ( \sum_{1 \leq h \leq q/2a} \|a h/q\|^{-1} + \sum_{q/2a \leq h \leq q/2} \|a h/q\|^{-1}).\]
 Therefore 
\[ |S|^2 \leq N + 4 q \sum_{1 \leq h < q/2}  1/h = N + O(q \log q), \] 
so that 
$ |S| =O( N^{1/2} + q^{1/2} (\log q)^{1/2}).$

We now turn to the case $N > q/4$. We may  write the integers $M \leq n < M+N$ as the union of $O(N(q/4)^{-1} + 1)$ blocks of at most $N'=q/4$ integers, and the first case then applies to the sum over each of these shorter blocks of integers. We then see that 
\[ |S| \leq O((N(q/4)^{-1} + 1)( N'^{1/2}  + q^{1/2}( \log q)^{1/2}) = O( (Nq^{-1/2} + q^{1/2})(\log q)^{1/2}).\]

\subsection*{Simultaneous Dirichlet approximation}
A standard reference is \cite{HarWri08}. 
Let a dimension $m \geq 1$ be fixed. Then every $y \in [0,1]^m$ can be approximated by $(a_1/q,\ldots, a_m/q)$ with $1 \leq a_1,\ldots, a_m \leq q$, $1 \leq q \leq Q$, with $|y_j - a_j/q| \leq (qQ^{1/m})^{-1}$ for each $1 \leq j \leq m$. To see this, fix an integer $P$ and divide the unit cube $[0,1]^m$ into smaller cubes of side-length $1/P$, of which there are $P^m$. We define a set of $P^m+1$ points in $[0,1]^m$ by considering the fractional part (that is, the value modulo 1) of $(ky_1,\ldots, ky_m)$ for each $0 \leq k \leq P^m$. By the pigeonhole principle, one of the smaller cubes must contain   two such points, say for the values $k' < k''$. Consequently, there exists some integral tuple $z \in \Z^m$ such that for each $1 \leq j \leq m$, $|(k''-k')y_j - z_j| \leq 1/P$. This yields 
$|y_j  - z_j/q| \leq (qP)^{-1}$, where $q = k'' - k' \leq P^m$, which suffices, with  $Q = P^m$. 
 
\section*{Funding Acknowledgements}
Pierce has been partially supported by NSF   CAREER grant DMS-1652173, a Sloan Research Fellowship, and the AMS Joan and Joseph Birman Fellowship. Pierce thanks the Hausdorff Center for Mathematics for providing a focused research environment during a portion of this work.

\bibliographystyle{alpha}
\bibliography{NoThBibliography}
\label{endofproposal}

\end{document}

%% file: format.tex
\numberwithin{equation}{section}

\newtheorem{thm}{Theorem}[section]
\newtheorem*{thm*}{Theorem}

\newtheorem{lemma}[thm]{Lemma}

\theoremstyle{remark}
\newtheorem{remark}[thm]{Remark}

\newcommand{\R}{\mathbb{R}}
\newcommand{\T}{\mathbb{T}}
\newcommand{\Z}{\mathbb{Z}}

\newcommand{\ep}{\varepsilon}

\newcommand{\con}{\equiv}

\newcommand{\ndiv}{\nmid}
\newcommand{\modd}[1]{\; ( \text{mod} \; #1)}
\newcommand{\bstack}[2]{#1 \atop #2}

\newcommand{\maps}{\rightarrow}

\newcommand{\Union}{\bigcup}
\newcommand{\union}{\cup}

\newcommand{\supp}{{\rm supp \;}}

\newcommand{\al}{\alpha}
\newcommand{\be}{\beta}

\newcommand{\del}{\delta}
\newcommand{\Del}{\Delta}

\newcommand{\om}{\omega}

\newcommand{\sig}{\sigma}
\newcommand{\lam}{\lambda}

\newcommand{\Bcal}{\mathcal{B}}

\newcommand{\beq}{\begin{equation}}
\newcommand{\eeq}{\end{equation}}

\makeatletter
\def\@tocline#1#2#3#4#5#6#7{\relax
  \ifnum #1>\c@tocdepth 
  \else
    \par \addpenalty\@secpenalty\addvspace{#2}%
    \begingroup \hyphenpenalty\@M
    \@ifempty{#4}{%
      \@tempdima\csname r@tocindent\number#1\endcsname\relax
    }{%
      \@tempdima#4\relax
    }%
    \parindent\z@ \leftskip#3\relax \advance\leftskip\@tempdima\relax
    \rightskip\@pnumwidth plus4em \parfillskip-\@pnumwidth
    #5\leavevmode\hskip-\@tempdima
      \ifcase #1
       \or\or \hskip 1em \or \hskip 2em \else \hskip 3em \fi%
      #6\nobreak\relax
    \hfill\hbox to\@pnumwidth{\@tocpagenum{#7}}\par
    \nobreak
    \endgroup
  \fi}
\makeatother